\documentclass[11pt,a4paper]{article}
\usepackage[utf8]{inputenc}
\usepackage[T1]{fontenc}
\usepackage[english,french]{babel}
\usepackage{fancyhdr}		
\usepackage{graphicx}
\usepackage{colortbl}
\usepackage{footnote}
\usepackage{makeidx}
\usepackage{amsfonts}
\usepackage{mathrsfs}
\usepackage{indentfirst}
\usepackage{shorttoc}
\usepackage[nottoc, notlof, notlot]{tocbibind}
\usepackage{graphicx}
\usepackage{amsmath, amsthm, amssymb}
\usepackage{geometry}
\usepackage{multicol}
\usepackage{multirow}
\usepackage[all]{xy}
\usepackage{chapterbib}
\usepackage{young}
\usepackage{stmaryrd}
\usepackage{tocbibind}
\usepackage{cite}
\usepackage{dsfont}
\usepackage{tikz}
\usepackage{MnSymbol}
\usetikzlibrary{shapes}
\usepackage{tikz-cd}
\usetikzlibrary{arrows}

\def\ker{\mathop{\rm ker}\nolimits}

\newtheoremstyle{plain}
{\topsep}% espace avant
{\topsep}% espace après
{\itshape}% police du corps du théorème
{}% indentation (vide pour rien, \parindent)
{\bfseries}% police du titre du théorème
{.}% ponctuation après le théorème
{ }% espace après le théorème (\newline = saut de ligne)
{\thmname{#1}\thmnumber{ \textup{#2}}\thmnote{~: \textup{#3}}}% spécification du titre du théorème
\theoremstyle{plain}
\newtheorem{theoreme}{Theorem}[section]

\newtheorem{proposition}[theoreme]{Proposition}
\newtheorem{lemma}[theoreme]{Lemma}
\theoremstyle{definition}
\newtheorem*{definition}{Definition}
\theoremstyle{definition}
\newtheorem*{remark}{Remark}
\theoremstyle{definition}

\theoremstyle{definition}

\theoremstyle{definition}

\usepackage{etoolbox}
\patchcmd{\thebibliography}{\chapter*}{\section*}{}{}

\title{\Large{\textbf{On differential operators on complete symmetric varieties of type $A_1$ and $A_2$}}}
\author{\textsc{Benoît Dejoncheere}\footnote{Université de Lyon, Université Claude Bernard Lyon 1, CNRS UMR 5208, Institut Camille Jordan, 43 bd. du 11 novembre 1918, F-69622 Villeurbanne Cedex, France (\texttt{dejoncheere@math.univ-lyon1.fr})}}
\date{}

\begin{document}

{\let\newpage\relax\maketitle}

\selectlanguage{english}
\begin{abstract}
In this paper, we will look at the algebra of global differential operators $D_X$ on wonderful compactifications $X$ of symmetric spaces $G/H$ of type $A_1$ and $A_2$. We will first construct a global differential operator on these varieties that does not come from the infinitesimal action of $\mathfrak{g}$. We will then focus on type $A_2$, where we will show that $D_X$ is an algebra of finite type, and that for any invertible sheaf ${\cal L}$ on $X$, $H^{0}(X,{\cal L})$ is either 0 or a simple left $D_{X,{\cal L}}$-module. Finally, we will show with the help of local cohomology that this is still true for higher cohomology groups $H^{i}(X,{\cal L})$.
\end{abstract}

\section{Introduction}
In order to answer a Kazhdan-Lusztig conjecture linking characters of some Verma modules and characters of irreducible highest weight modules, Beilinson and Bernstein in \cite{BB} and Brylinski and Kashiwara in \cite{BK} have found out in the early 80's interesting properties on the sheaf of differential operators ${\cal D}_X$ on flag varieties $X$, and on its global sections algebra $D_X$. For example, we know that there is an equivalence of categories between left ${\cal D}_X$-modules which are quasi-coherent as ${\cal O}_X$-modules, and left $D_X$-modules (and we say that $X$ is ${\cal D}$-affine). An other example is that when $X$ is a flag variety, the morphism $L:\mathfrak{U}(\mathfrak{g})\to D_X$ given by the infinitesimal action of $\mathfrak{g}$ is surjective, and its kernel has an explicit description. This was more precisely investigated by Borho and Brylinski in \cite{BoB1}, \cite{BoB2} and \cite{BoB3}. \\
\indent
In the 90's, differential operators on toric varieties have been studied (for example in \cite{Mu} and \cite{Jo}) using combinatorial tools and a combinatorial description of toric varieties, and are rather well understood. But except for flag varieties and projective toric varieties, differential operators on projective varieties are not well understood. In this paper, we will look at the algebra of global differential operators $D_Y$ on some wonderful compactifications $Y$ of symmetric spaces $G/H$ of small positive rank. These varieties are wonderful varieties, which are natural generalization of flag varieties, since they are smooth projective varieties with a $G$-action with good properties. In particular, flag varieties $G/P$ are exactly the $G$-wonderful varieties of rank 0.\\
\indent
In the section 2, we will construct a differential operator on wonderful compactifications $Y$ of symmetric spaces $G/H$ with reduced root system $\tilde{\Phi}$ of type $A_1$ and $A_2$ which does not lie in the image of $L:\mathfrak{U}(\mathfrak{g})\to D_Y$. When $\tilde{\Phi}$ is of type $A_1$, we will see that this is actually not a surprise, since these varieties are actually flag varieties for a group $G'$ bigger than $G$, and we will then focus on some cases of type $A_2$.\\
\indent
We will use a result of Thaddeus (cf. \cite{Th2}) to see the wonderful compactifications $Y$ of $\textrm{PGL}_n$, $\textrm{PGL}_n/\textrm{PSO}_n$ and of $\textrm{PGL}_{2n}/\textrm{PSp}_{2n}$ as direct limit of GIT quotients of some Grassmannian $X$ by a $\mathbb{C}^*$-action. In section 3, we will show that this direct limit of GIT quotients is exactly the GIT quotient for an ample line bundle $\overline{\cal L}$ on $X$ with trivial $\mathbb{C}^*$-linearization when $n=3$. Using this description, we will be able to show that $D_Y$ is of finite type, and that the $H^0(Y,{\cal L})$ are either 0 or simple left $D_{Y,{\cal L}}$-modules, where ${\cal D}_{Y,{\cal L}}$ is the sheaf of differential operators of ${\cal L}$, and $D_{Y,{\cal L}}$ is the algebra of its global sections.\\
\indent
In these proofs, we will use the fact that the set of unstable points is of codimension at least two, and that allows us to extend uniquely sections of locally free sheaves on the set of semi-stable points $X^{ss}$ to sections of locally free sheaves on $X$. However, in higher cohomology we do not have $H^i(X,\overline{\cal L})=H^i(X^{ss},\overline{\cal L}_{|X^{ss}})$ anymore. In section 4, we will recall a few properties about local cohomology, and we will define generalized Cousin complexes in order to have an approximation of the cohomology group $H^i_{Z}(X,\overline{\cal L})$ when $Z$ is a Schubert variety of codimension $i$ in $X$. Using this complex, we will be able to show that the higher cohomology groups $H^i(Y,{\cal L})$ are either 0 or simple left $D_{Y,{\cal L}}$-modules.
\\

\textit{Acknowledgements} : I would like to thank Alexis Tchoudjem for his many useful remarks.

\section{Construction of a differential operator}
\subsection{Setup}
If $G$ is a connected reductive group, we say a $G$-variety $X$ is wonderful (cf. \cite{Lu}) if it is smooth, complete, projective, with a unique open $G$-orbit $\Omega$, and if $X\setminus\Omega$ is a union of $r$ prime $G$-stable divisors $X_1,\ldots X_r$ with normal crossings, such that for all $x,y$ in $X$, $x$ and $y$ are in the same $G$-orbit if and only if they belong to the same $X_i$, and that the intersection of all $X_i$'s is non-empty. We call $r$ the rank of the wonderful variety $X$.\\
\indent
Let $G$ be a connected semi-simple adjoint algebraic group, and $\theta:G\to G$ an involution of $G$. As constructed in \cite{DCP}, the symmetric space $G/G^{\theta}$ can be uniquely embedded into a wonderful $G$-variety $X$ with an open orbit isomorphic to $G/G^{\theta}$. It is done by choosing a good highest weight module $V_{\lambda}$ with a special regular weight $\lambda$ (in the sense of \cite{DCP}), and a nonzero $h\in V_{\lambda}^{G^{\theta}}$, and $X$ is the closure of $G.[h]$ in $\mathbb{P}(V_{\lambda})$. Let un choose a maximal $\theta$-split torus $T$ (ie. a maximal torus such that the dimension $r$ of $T_1:=\{x\in X|\theta(x)=x^{-1}\}$ is maximal), and let $B$ be a Borel subgroup containing $T$ such that each positive root which is not fixed by $\theta$ is sent to a negative one. Let us denote by $\Phi_0$ the subset of $\theta$-fixed roots in the root system $\Phi(G,T)$, and by $\Phi_1$ its complement. One can show that $X$ is a wonderful variety of rank $r$. The choice of $B$ and $T$ gives rise to the little root system $\tilde{\Phi}$, which is the (eventually non-reduced) root system of restricted roots in $\Phi_1$ to $T_1$. The restrictions to $T_1$ of simple non-$\theta$-fixed roots are the simple roots of $\tilde{\Phi}$, and let $\gamma_1,\ldots,\gamma_r$ be the double of these simple restricted roots.\\
\indent
One can construct an open $B$-cell $X_0\subset X$ that is isomorphic to the affine space $U'\times\mathbb{A}^r$ where $U'=\prod\limits_{\alpha\in\Phi_1^+}U_{\alpha}$, and where $\mathbb{A}^r$ is the closure of $T_1.[h]$, which is embedded with the morphism
\[
\begin{tabular}{rcccl}
$T_1$&$\to $&$T_1.[h]$&$\hookrightarrow $&$\mathbb{A}^r$\\
$t$&$\mapsto $&$t.[h]$&$\mapsto $&$(\gamma_1(t),\ldots,\gamma_r(t))$
\end{tabular}
\]
The aim of this section is to construct differential operators on $X_0$ that are restrictions to $X_0$ of global differential operators that do not come from the infinitesimal action of $\mathfrak{g}$ when $\tilde{\Phi}$ is of type $A_1$ or $A_2$.
\subsection{Satake diagrams}
Like semi-simple Lie algebras which can be classified with Dynkin diagrams, symmetric spaces can be classified by the Satake diagrams, which we are going to recall the definition :
\begin{definition}
Let $G$ be a semi-simple algebraic group and $\theta:G\to G$ an involution. The Satake diagram of the symmetric space $G/G^{\theta}$ is a diagram constructed from the Dynkin diagram of $G$ by adding :
\item[(1)] a coloration on vertices : black for vertices representing roots in $\Phi_0$, white for the other ones ;
\item[(2)] two-headed arrows between vertices representing different simple roots $\alpha_i$ and $\alpha_j$ such that
\[
-\theta(\alpha_i)=\alpha_j+\sum\limits_{\beta_i\in\Phi_0}n_i\beta_i
\]
for $n_i\geq 0$.
\end{definition}
One can notice that each white vertex is linked to at most one white vertex with two-headed arrows (cf. \cite{DCP} 1.3%?
). The Satake diagram of the symmetric space $G/H$ is useful to read easily the action of $\theta$ on roots. This action being linear, it is enough to understand how it acts on simple roots :
\begin{proposition}
Let $\alpha$ be a simple root.
\item[(1)] if $\alpha$ is a black root, then $\theta(\alpha)=\alpha$ ;
\item[(2)] if $\alpha$ is a white root linked to no other white root with two-headed arrows, then $-\theta(\alpha)$ is the highest root which can be written as $\alpha+\sum\limits_{\beta_i\in\Phi_0}n_i\beta_i$ ;
\item[(3)] if $\alpha$ is a white root linked to an other white root $\alpha'$ with a two-headed arrow, then $-\theta(\alpha)$ is the highest root which can be written as $\alpha'+\sum\limits_{\beta_i\in\Phi_0}n_i\beta_i$. Moreover, $\alpha-\alpha'-\theta(\alpha)$ is the highest root which can be written as $\alpha+\sum\limits_{\beta_i\in\Phi_0}n_i\beta_i$.
\end{proposition}
\begin{proof}
(1) is just the definition of a black root. To show (2), let us write
\[
-\theta(\alpha)=\alpha+\sum\limits_{\beta_i\in\Phi_0}n_i\beta_i
\]
Let $\delta=\alpha+\sum n'_i\beta_i$ be a highest root which can be written like this such that $\delta+\theta(\alpha)$ has non-negative coefficients. Then $n'_i\geq n_i$, and since $\delta\in\Phi_1$ is positive, $-\theta(\delta)$ is negative, so $n_i\geq n'_i$. Hence $-\theta(\alpha)=\delta$, and $\delta$ is the unique highest root which can be written as $\delta=\alpha+\sum n'_i\beta_i$. The proof of (3) is similar, using the symmetry between $\alpha$ and $\alpha'$.
\end{proof}
Thanks to Satake diagrams, we can also show that we can reduce ourselves without loss of generality to the following cases :\\
\begin{itemize}
\item[(1)] $G/H$ with simple $G$ ;
\item[(2)] $G\times G/G$ with simple $G$, and $\theta : (g_1,g_2)\mapsto (g_2,g_1)$
\end{itemize}
\begin{proposition}
If $G$ is semi-simple and $H=G^{\theta}$, the Satake diagram $S$ of $G/H$ is a disjoint union of Satake diagrams of $G'/H'$ with simple $G'$ and $G''\times G''/G''$ with simple $G''$
\end{proposition}
\begin{proof}
Let us call \textit{connected component} of $S$ a connected component of the underlying Dynkin diagram of $G$ (we forget two-headed arrows), and \textit{strictly connected component} a connected component of $S$ (where two-headed arrows are not forgotten). Let $\alpha$ a white arrow such that $-\theta(\alpha)=\alpha'+\sum n_i\beta_i$ and $\alpha'$ is not in the connected component of $\alpha$. Since $-\theta(\alpha)$ is a root, the $\beta_i$ lie in the connected component of $\alpha'$. The role of $\alpha$ and $\alpha'$ being symmetric, they lie in the connected component of $\alpha$ as well. Hence $-\theta(\alpha)=\alpha'$, and $\alpha$ and $\alpha'$ are only linked to white roots. The Killing form being $\theta$-invariant, the Cartan matrices of the connected components of $\alpha$ and $\alpha'$ are the same, and by fixing a numerotation on the simple roots of a simple $G''$ associated to this Cartan matrix, we get $-\theta(\alpha_i)=\alpha'_i$. Hence the strictly connected component of $\alpha$ is the Satake diagram of $G''\times G''/G''$. If in a given connected component such an $\alpha$ does not exist, then it is a strictly component of $S$, and it corresponds to the Satake diagram of $G'/H'$ for a simple $G$.
\end{proof}
\subsection{Construction when $-\theta$ has no fixed white roots}
Let us assume that $-\theta$ has no fixed simple white roots. This hypothesis allows us to lighten the computations on $X=\overline{G/H}$ since :
\begin{proposition}
If $\alpha\neq-\theta(\alpha)$, then $[X_{\alpha},X_{\theta(\alpha)}]=0$. In particular, $e^{X_{\alpha}}$ and $e^{X_{\theta(\alpha)}}$ are commuting in $G$, and if $h\in X$ is in the open $G$-orbit such that its stabilizer $G_h=H$, then $e^{X_{\alpha}}.h=e^{-X_{\theta(\alpha)}}.h$
\end{proposition}
\begin{proof}
By hypothesis, we already know that $\alpha+\theta(\alpha)\neq 0$. To show that it is not a root, we use the lemma (cf. \cite{DCP} 1.3) :
\begin{lemma}
If $\alpha\in\Phi_0$, then $\theta$ acts as the identity on $\mathfrak{g}_{\alpha}$.
\end{lemma}
Hence if $\alpha+\theta(\alpha)\neq 0$ is a root, it lies in $\Phi_0$, and $[X_{\alpha},X_{\theta(\alpha)}]\neq 0$. Since $\theta([X_{\alpha},X_{\theta(\alpha)}])=-[X_{\alpha},X_{\theta(\alpha)}]$, we have a contradiction. By Campbell-Hausdorff formula, we get the commutativity of $e^{X_{\alpha}}$ and $e^{X_{\theta(\alpha)}}$.
\end{proof}
This condition can be easily checked thanks to the Satake diagram : it just means that each white root is either linked to at least a black root, or connected to a white root by a two-headed arrow.\\
\indent
Since we have an isomorphism $\varphi : U'\times\mathbb{A}^r\to X_0$, by fixing coordinates $x_1,\ldots x_s$ on $U'$ we can get coordinates on $X_0$, and $k[X_0]=k[x_1,\ldots,x_s,t_1,\ldots,t_r]$. This choice of coordinates can be done by fixing a total order $\tilde{\prec}$ on $\Phi_1^+$. Let us write the roots in $\Phi_1^+$ $\tilde{\alpha}_1,\ldots,\tilde{\alpha}_s$ with $\tilde{\alpha}_i\tilde{\prec}\tilde{\alpha}_{i+1}$. We can represent a point $x$ in $X_0$ by a couple $(\prod u_{\tilde{\alpha}_i}(x_{\tilde{\alpha}_i}),(t_1,\ldots,t_r))$ such that on the open set $X_0\cap\Omega$ we have
\[
\varphi(\prod u_{\tilde{\alpha}_i}(x_{\tilde{\alpha}_i}),(t_1,\ldots,t_r))=\prod u_{\tilde{\alpha}_i}(x_{\tilde{\alpha}_i})a(t_1,\ldots,t_r).h=x
\]
where $a(t_1,\ldots,t_r)\in T_1/\textrm{Stab}_{T_1}(h)$ is such that $\gamma_i(a(t_1,\ldots,t_r))=t_i$. If we have an other total order $\prec$ on $\Phi_1^+$, we can also represent $x$ by a couple $(\prod u_{\alpha_i}(x_{\alpha_i}),(t_1,\ldots,t_r))$. Then we have
\[
k[X_0]=k[x_{\tilde{\alpha}_i},t_j]=k[x_{\alpha_i},t_j]
\]
and $x_{\alpha_i}$ are polynomials in the $x_{\tilde{\alpha}_i}$ and $t_j$. Moreover, the Weyl algebras
\[
k[x_{\tilde{\alpha}_i},t_j,\partial_{x_{\tilde{\alpha}_i}},\partial_{t_j}]=k[x_{\alpha_i},t_j,\partial_{x_{\alpha_i}},\partial_{t_j}]
\]
are isomorphic, and the $\partial_{x_{\alpha_i}}$ are polynomials in the $x_{\tilde{\alpha}_i},t_j,\partial_{x_{\tilde{\alpha}_i}}$ and $\partial_{t_j}$. Hence for all white root $\alpha$, we will choose a total order $\prec_{\alpha}$ on $\Phi_1^+$ which makes the computations easier. We will now state our results :
\begin{theoreme}
Let $G$ be a simple connected algebraic group of adjoint type, let $X$ be its wonderful compactification, and $T\subset B$ chosen as in 2.1. The following are equivalent :
\item[(a)] There exists a differential operator on the affine $B\times B^-$-cell $X_0$, which is a monomial in the $\partial_{t_i}$, that is the restriction to $X_0$ of a global differential operator that does not lie in the image of $\mathfrak{U}(\mathfrak{g})$ ;
\item[(b)] $\Phi$ is of type $A_1$ or $A_2$.
\end{theoreme}
\begin{proof}
Let $U$ be the unipotent radical of $B$. We have an isomorphism
\[
\phi : U\times U'\times \mathbb{A}^r\to X_0
\]
Let $\prec_i$ a total relation on $\Phi^+$ such that $\alpha_i$ is maximal for this relation, and denote by $\delta_1\prec_i\delta_2\prec_i\ldots\prec_i\delta_q$ the other elements of $\Phi^+$. We can represent a point in $X_0$ by a couple
\[
\phi(u_{\delta_1}(x_{\delta_1})\ldots u_{\delta_q}(x_{\delta_q})u_{\alpha_i}(x),u_{\delta'_1}(y_{\delta_1})\ldots u_{\delta'_q}(y_{\delta_q})u_{\alpha'_i}(y),(t_1,\ldots,t_r))
\]
and a point in $n_{\alpha_i}.X_0$ by
\[
n_{\alpha_i}\phi(u_{\delta_1}(x'_{\delta_1})\ldots u_{\delta_q}(x'_{\delta_q})u_{\alpha_i}(x'),u_{\delta'_1}(y'_{\delta_1})\ldots u_{\delta'_q}(y'_{\delta_q})u_{\alpha'_i}(y'),(t,_1,\ldots,t,_r))
\]
On the open set $(\prod t'_i\neq 0,x'\neq 0)$, we get
\begin{changemargin}{-1cm}{0cm}
\[
\begin{aligned}
&n_{\alpha_i}.\phi(u_{\delta_1}(x'_{\delta_1})\ldots u_{\delta_q}(x'_{\delta_q})u_{\alpha_i}(x')u_{\delta'_1}(y'_{\delta_1})\ldots u_{\delta'_q}(y'_{\delta_q})u_{\alpha'_i}(y'),(t'_1,\ldots,t'_r))\\
=&u_{s_{\alpha_i}(\delta_1)}(k_1x'_{\delta_1})\ldots u_{s_{\alpha_i}(\delta_q)}(k_qx'_{\delta_q})n_{\alpha_i}u_{\alpha_i}(x')u_{\delta'_1}(y'_{\delta_1})\ldots u_{\delta'_q}(y'_{\delta_q})u_{\alpha'_i}(y')a(t'_1,\ldots,t'_r).[h]\\
=&u_{s_{\alpha_i}(\delta_1)}(k_1x'_{\delta_1})\ldots u_{s_{\alpha_i}(\delta_q)}(k_qx'_{\delta_q})\alpha_i^{\vee}(-x'^{-1})u_{\alpha_i}(-x')u_{-\alpha_i}(x'^{-1})u_{\delta'_1}(y'_{\delta_1})\ldots u_{\delta'_q}(y'_{\delta_q})u_{\alpha'_i}(y')a(t'_1,\ldots,t'_r).[h]\\
=&u_{s_{\alpha_i}(\delta_1)}(k_1x'_{\delta_1})\ldots u_{s_{\alpha_i}(\delta_q)}(k_qx'_{\delta_q})\alpha_i^{\vee}(-x'^{-1})u_{\alpha_i}(-x')u_{\delta'_1}(y'_{\delta_1})\ldots u_{\delta'_q}(y'_{\delta_q})u_{\alpha'_i}(y')u_{-\alpha_i}(x'^{-1})a(t'_1,\ldots,t'_r).[h]\\
=&u_{s_{\alpha_i}(\delta_1)}(k_1x'_{\delta_1})\ldots u_{s_{\alpha_i}(\delta_q)}(k_qx'_{\delta_q})\alpha_i^{\vee}(-x'^{-1})u_{\alpha_i}(-x')u_{\delta'_1}(y'_{\delta_1})\ldots u_{\delta'_q}(y'_{\delta_q})u_{\alpha'_i}(y')u_{\alpha'_i}(-t'_1x'^{-1})a(t'_1,\ldots,t'_r).[h]\\
=&u_{\delta_1}(P_1(x'_{\delta_i}))\ldots u_{\delta_q}(P_q(x'_{\delta_i}))u_{\alpha_i}(-x'^{-1})u_{\delta'_1}(y'_{\delta_1})\ldots u_{\delta'_q}(y'_{\delta_q})u_{\alpha'_i}(y'-t'_1x'^{-1})\alpha_i^{\vee}(-x'^{-1})a(t'_1,\ldots,t'_r).[h]\\
=&u_{\delta_1}(P_1(x'_{\delta_i}))\ldots u_{\delta_q}(P_q(x'_{\delta_i}))u_{\alpha_i}(-x'^{-1})u_{\delta'_1}(y'_{\delta_1})\ldots u_{\delta'_q}(y'_{\delta_q})u_{\alpha'_i}(y'-t'_1x'^{-1})a(t_1(-x')^{-\langle\gamma_1,\alpha_i\rangle},\ldots,t'_r(-x')^{\langle\gamma_r,\alpha_i\rangle}).[h]\\
\end{aligned}
\]
\end{changemargin}
with $k_j$ such that $ad(n_{\alpha_i}).X_{\delta_j}=k_jX_{s_{\alpha_i}(\delta_j)}$, and $P_j$ are polynomials in the $x'_{\delta_k}$, of degree at most one in each of its variables, which are given by the Campbell-Hausdorff formula. These changes of coordinates can be extended when the $t'_j$'s are 0, and we get on $n_{\alpha_i}.X_0\cap X_0=(x'\neq 0)$
\[
\begin{array}{rcl}
\partial_{t_i}&=&(-x')^{\langle\gamma_i,\alpha_i\rangle}\partial_{t'_i}-(-x')^{\langle\gamma_i,\alpha_i\rangle-1}\partial_{y'}\\
\partial_{t_j}&=&(-x')^{\langle\gamma_j,\alpha_i\rangle}\partial_{t'_j} \textrm{ pour }j\neq i
\end{array}
\]
\indent
Let $\partial=\prod\partial_{t_i}^{n_i}$ be a differential operator on $X_0$. Let us assume that $\partial$ is the restriction of a differential operator on each $s_{\alpha}.X_0\cup X_0$ for $\alpha$ simple root. Then $\partial$ is the restriction of a differential operator defined on an open subset which is $s_{\alpha}Bs_{\alpha}$-stable for each simple $\alpha$, hence it can be extended to a $G$-stable open subset $\Omega$. Since $X_0$ intersects all $G$-orbits of $X$, $\Omega=X$, and $\partial$ is the restriction to $X_0$ of a global differential operator.\\
\indent
Since for all $i$ and $j$, we have
\[
\langle\gamma_j,\alpha_i\rangle=\langle\gamma_j,\alpha'_i\rangle=\langle\alpha_j,\alpha_i\rangle
\]
and $\partial$ is the restriction of a global differential operator if and only if for all $1\leq i\leq r$, we have
\[
n_i+\sum\limits_{j\neq i}n_j\langle \alpha_j,\alpha_i\rangle\geq 0
\]
If $C$ denotes the Cartan matrix of $\Phi$, this can be restated as
\[
(C-I)(n_1,\ldots,n_r)\in(\mathbb{Z}_{\geq 0})^r
\]
This is possible if and only if $r\leq 2$, with $\Phi$ of type $A$, and the $n_i$'s being equal. The operators $\partial_{t_1}$ in type $A_1$ and $\partial_{t_1}\partial_{t_2}$ in type $A_2$ are such operators.\\
\indent
$\partial_{t_1}\partial_{t_2}$ does not lie in the image of $\mathfrak{U}(\mathfrak{g})$ because it does not stabilize $\Gamma(X_0,{\cal I}_{D_1}{\cal I}_{D_2})$, since $t_1t_2\in\Gamma(X_0,{\cal I}_{D_1}{\cal I}_{D_2})$, and $\partial_{t_1}\partial_{t_2}.t_1t_2=1\not\in\Gamma(X_0,{\cal I}_{D_1}{\cal I}_{D_2})$. The same holds for $\partial_{t_1}$ not stabilizing $\Gamma(X_0,{\cal I}_{D_1})$.
\end{proof}
A similar statement holds for wonderful compactifications of symmetric spaces $G/H$ with simple $G$ :
\begin{theoreme}
Let $X=\overline{G/H}$ be a symmetric space with a simple connected $G$ of adjoint type. Let us assume moreover that any simple root in $\Phi$ is not fixed by $-\theta$. The following are equivalent :
\item[(a)] There exists a differential operator on the affine $B$-cell $X_0$, which is a monomial in the $\partial_{t_i}$, that is the restriction to $X_0$ of a global differential operator that does not lie in the image of $\mathfrak{U}(\mathfrak{g})$ ;
\item[(b)] $\tilde{\Phi}$ is of type $A_1$ or $A_2$.
\end{theoreme}
\begin{proof}
let $\alpha$ be a white root such that $\alpha-\theta(\alpha)=\gamma_i$. Let us assume first that $\langle\theta(\alpha),\alpha^{\vee}\rangle=0$. Then we get that $-\theta(\alpha)+\alpha\not\in\Phi$ since $s_{\alpha}(\theta(\alpha))=\theta(\alpha)$ and $s_{\alpha}(\theta(\alpha)-\alpha)=\theta(\alpha)+\alpha\not\in\Phi$. Let us fix $\prec$ a partial order on $\Phi_1^+$ such that
\[
\beta_1\prec\ldots\prec\beta_q\prec\delta_1\prec\ldots\prec\delta_p\prec-\theta(\alpha)\prec\alpha
\]
where the $\beta_i$'s (resp. $\delta_i$'s) are roots in $\Phi_1^+\setminus\{-\theta(\alpha),\alpha\}$ such that $s_{\alpha}(\beta)\not\in\Phi_0$ (resp. $s_{\alpha}(\delta)\in\Phi_0$). Let $U$ be the open $(\prod t'_i\neq 0,x'\neq 0)$ of $n_{\alpha}.X_0$. We have on $U$
\[
\begin{aligned}
&n_{\alpha}.\varphi(u_{\beta_1}(x'_{\beta_1})\ldots u_{\beta_q}(x'_{\beta_q})u_{\delta_1}(x'_{\delta_1})\ldots u_{\delta_p}(x'_{\delta_p})u_{-\theta(\alpha)}(y')u_{\alpha}(x'),(t'_1,\ldots,t'_r))\\
=&u_{\beta_1}(Q_1(x'_{\beta_j},x'_{\delta_j},x'^{-1},y'-x'^{-1}t_i))\ldots u_{\beta_q}(Q_q(x'_{\beta_j},x'_{\delta_j},x'^{-1},y'-x'^{-1}t_i)) \\ &u_{\delta_1}(P_1(x'_{\delta_j},x'^{-1},y'-x'^{-1}t_i))\ldots u_{\delta_p}(P_p(x'_{\delta_j},x'^{-1},y'-x'^{-1}t_i)) \\ &u_{-\theta(\alpha)}(y'-x'^{-1}t_i)u_{\alpha}(-x'^{-1})a(t'_1(-x')^{\langle\gamma_1,\alpha_i\rangle},\ldots,t'_r(-x')^{\langle\gamma_r,\alpha_i\rangle}).[h]\\
\end{aligned}
\]
with $P_j$'s and $Q_j$'s obtained the same way than before. These change of coordinates can be extended to $n_{\alpha}.X_0\cap X_0=(x'\neq 0)$ and we get
\[
\begin{array}{rcl}
\partial_{t_i}&=&(-x')^{\langle\gamma_i,\alpha\rangle}\partial_{t'_i}-(-x')^{\langle\gamma_i,\alpha\rangle-1}\partial_{y'}\\
\partial_{t_j}&=&(-x')^{\langle\gamma_j,\alpha\rangle}\partial_{t'_j} \textrm{ pour }j\neq i
\end{array}
\]
Now let us assume that $\langle\theta(\alpha),\alpha^{\vee}\rangle\neq 0$. Let $\prec'$ be a total order on $\Phi_1^+$ such that
\[
\beta_1\prec'\ldots\prec'\beta_q\prec'\delta_1\prec'\ldots\prec'\delta_p\prec's_{\alpha}(-\theta(\alpha))\prec'-\theta(\alpha)\prec'\alpha
\]
where the $\beta_i$'s (resp. $\delta_i$'s) are roots in $\Phi_1^+\setminus\{s_{\alpha}(-\theta(\alpha)),-\theta(\alpha),\alpha\}$ such that $s_{\alpha}(\beta)\not\in\Phi_0$ (resp. $s_{\alpha}(\delta)\in\Phi_0$). We get
\begin{changemargin}{-1cm}{0cm}
\[
\begin{aligned}
&n_{\alpha}.\varphi(u_{\beta_1}(x'_{\beta_1})\ldots u_{\beta_q}(x'_{\beta_q})u_{\delta_1}(x'_{\delta_1})\ldots u_{\delta_p}(x'_{\delta_p})u_{s_{\alpha}(-\theta(\alpha))}(z')u_{-\theta(\alpha)}(y')u_{\alpha}(x'),(t'_1,\ldots,t'_r))\\
=&u_{\beta_1}(Q_1(x'_{\beta_j},x'_{\delta_j},x'^{-1},y',z'-x'^{-1}t_i))\ldots u_{\beta_q}(Q_q(x'_{\beta_j},x'_{\delta_j},x'^{-1},y',z'-x'^{-1}t_i)) \\ &u_{\delta_1}(P_1(x'_{\delta_j},x'^{-1},y',z'-x'^{-1}t_i))\ldots u_{\delta_p}(P_p(x'_{\delta_j},x'^{-1},y',z'-x'^{-1}t_i))u_{s_{\alpha}(-\theta(\alpha))}(Q_{q+1}(x'_{\beta_j},x'_{\delta_j},x'^{-1},y',z'-x'^{-1}t_i)) \\ &u_{-\theta(\alpha)}(z'-x'^{-1}t_i)u_{\alpha}(-x'^{-1})a(t'_1(-x')^{\langle\gamma_1,\alpha_i\rangle},\ldots,t'_r(-x')^{\langle\gamma_r,\alpha_i\rangle}).[h]
\end{aligned}
\]
\end{changemargin}
with similar $P_j$'s and $Q_j$'s. On $n_{\alpha}.X_0\cap X_0$, we have
\[
\begin{array}{rcl}
\partial_{t_i}&=&(-x')^{\langle\gamma_i,\alpha\rangle}\partial_{t'_i}-(-x')^{\langle\gamma_i,\alpha\rangle-1}\partial_{z'}\\
\partial_{t_j}&=&(-x')^{\langle\gamma_j,\alpha\rangle}\partial_{t'_j} \textrm{ pour }j\neq i
\end{array}
\]
In both cases, the differential operator $\partial=\prod\partial_{t_i}^{n_i}$ is the restriction of a global differential operator on $X$ if and only if for all family of roots $(\alpha_1,\ldots,\alpha_r)$ such that $\alpha_i-\theta(\alpha_i)=\gamma_i$, we have
\[
(C-I)(n_1,\ldots,n_r)\in(\mathbb{Z}_{\geq 0})^r
\]
where $C$ is the matrix $C=(\langle \gamma_j,\alpha_i\rangle)$. But $C$ is exactly the Cartan matrix of $\tilde{\Phi}$ when $\tilde{\Phi}$ is not of type $BC_n$, and it is the matrix
\[
(\langle\alpha_i,\gamma_j\rangle)_{1\leq i,j\leq r}=
\begin{pmatrix}
2&-1&\ldots&0&0\\
-1&2&\ldots&0&0\\
\ldots&\ldots&\ldots&\ldots&\ldots\\
0&0&\ldots&2&-1\\
0&0&\ldots&-1&1
\end{pmatrix}
\]
when $\tilde{\Phi}$ is of type $BC_n$. Hence $\partial$ is the restriction of a global differential operator if and only if $\tilde{\Phi}$ is of type $A_1$ (and $\partial=\partial_{t_1}^{n_1}$), or $\tilde{\Phi}$ is of type $A_2$ (and $\partial=(\partial_{t_1}\partial_{t_2})^n$). For the same reasons as before, these operators do not lie in the image of $\mathfrak{U}(\mathfrak{g})$.
\end{proof}
\begin{remark}
By looking at the classification of symmetric spaces (cf. \cite{He} or \cite{Ti}), the symmetric spaces $G/H$ with simple connected $G$ of adjoint type of type $A_1$ or $A_2$ such that $-\theta$ has no fixed point in $\Phi$ are exactly the following :
\item[(1)] $\textrm{PGL}_6/\textrm{PSp}_6$ ;
\item[(2)] $E_6/F_4$ ;
\item[(3)] $\textrm{PGL}_n/\textrm{GL}_{n-1}$ with $n\geq 3$ ;
\item[(3')] $\textrm{PSO}_6/\textrm{GL}_3\simeq \textrm{PGL}_4/\textrm{GL}_3$ ;
\item[(4)] $\textrm{Psp}_{2n}/\textrm{P}(\textrm{SL}_2\times\textrm{Sp}_{2n-2})$ with $n\geq 2$ ;
\item[(5)] $\textrm{PSO}_n/\textrm{P}(\textrm{SO}_{n-1}\times k^*)$ with $n\geq 5$ ;
\item[(5')] $\textrm{PGL}_4/\textrm{PSp}_4\simeq \textrm{PSO}_6/\textrm{P}(\textrm{SO}_5\times k^*)$ ;
\item[(6)] $F_4/\textrm{PSO}_9$.\\
The two remaining cases of type $A_1$ or $A_2$ are $\textrm{PGL}_2/k^*$ and $\textrm{PGL}_3/\textrm{PSO}_3$, which have to be investigated. In these cases, $-\theta$ is the identity on $\Phi$.
\end{remark}
If $\alpha$ is fixed by $-\theta$, the computations are more complicated, since $e^{X_{\alpha}}$ and $e^{X_{\theta(\alpha)}}$ are not commuting any more, and to get rid of the $u_{-\alpha}(x).[h]$, we actually have to find some $u_{\alpha}(y)\alpha^{\vee}(z)u_{-\alpha}(-x)$ in $H$.\\
\textbullet \indent Case of $\textrm{PGL}_2/k^*$. Let $y=x'^{-1}\alpha(t'_1)$ and let $\xi$ such that $\xi^2=1+y^2$. On the open $(x'\neq 0, t'_1\neq 0, 1+y^2\neq 0)$ we have
\[
\begin{aligned}
n_{\alpha}u_{\alpha}(x')a(t'_1).[h]&=\alpha^{\vee}(-x'^{-1})u_{\alpha}(-x')a(t'_1)u_{-\alpha}(y).[h]\\
&=\alpha^{\vee}(-x'^{-1})u_{\alpha}(-x')a(t'_1)\alpha^{\vee}(\xi^{-1})u_{\alpha}(y).[h]\\
&=\alpha^{\vee}(-x'^{-1})u_{\alpha}(-x'+\frac{x'^{-1}t'_1}{1+x'^{-2}t'_1})a(\frac{t'_1}{(1+x'^{-2}t'_1)^2}).[h]\\
&=u_{\alpha}(\frac{-x'}{x'^2+t'_1})a(\frac{t'_1}{(x'^2+t'_1)^2}).[h]
\end{aligned}
\]
these change of coordinates can be extended to $n_{\alpha}.X_0\cap X_0=(x'\neq 0)$, and we get
\[
\partial_{t_1}=(x'^2-t'_1)(x'^2+t'_1)\partial_{t'_1}-x'(x'^2+t'_1)\partial_{x'}
\]
Hence $\partial_{t_1}$ is the restriction of a global differential operator, which does not lie in the image of $\mathfrak{U}(\mathfrak{g})$.\\
\textbullet \indent Case of $\textrm{PGL}_3/\textrm{PSO}_3$. Let $\psi=x'^{-1}\alpha_{1}(a(t'_{1},t'_{2}))$, and let $\xi$ such that $\xi^{2}=(1+\psi^{2})$. Then we have
\[
\begin{pmatrix}
\xi-\xi^{-1}\psi^{2}&-\xi^{-1}\psi&0\\
\xi^{-1}\psi&\xi^{-1}&0\\
0&0&1
\end{pmatrix}
\in\text{SO}_{3}
\]
and in the open $x'\neq 0, t'_1t'_2\neq 0, 1+\psi^2\neq 0)$ we have
\[
\begin{aligned}
&(n_{\alpha_{1}}.(u_{\alpha_{1}+\alpha_{2}}(z')u_{\alpha_{2}}(y')u_{\alpha_{1}}(x')a(t'_{1},t'_{2}))).[h]&\\
=&u_{\alpha_{1}+\alpha_{2}}(y')u_{\alpha_{2}}(-z')\alpha_{1}^{\vee}(-x'^{-1})u_{\alpha_{1}}(-x')a(t'_{1},t'_{2})e^{\psi X_{-\alpha_{1}}}.[h]&\\
=&u_{\alpha_{1}+\alpha_{2}}(y')u_{\alpha_{2}}(-z')u_{\alpha_{1}}(\frac{-x'}{x'^{2}+t'_{1}})a(\frac{t'_{1}}{(x'^{2}+t'_{1})^{2}},t'_{2}(x'^{2}+t'_{1})).[h]&
\end{aligned}
\]
This can be extended to $n_{\alpha_1}.X_0\cap X_0$, and we get
\[
\begin{array}{rcl}
\partial_{t_1}&=&(x'^2-t'_1)(x'^2+t'_1)\partial_{t'_1}+t'_2(x'^2+t'_1)\partial_{t'_2}-x'(x'^2+t'_1)\partial_{x'}\\
\partial_{t_{2}}&=&\frac{-1}{x'^{2}+t'_{1}}\partial_{t'_{2}}\\
\partial_{t_{1}}\partial_{t_{2}}&=&(x'^{2}-t'_{1})\partial_{t'_{1}}\partial_{t'_{2}}+t'_{2}\partial_{t'_{2}}^{2}-x'\partial_{x'}\partial_{t'_{2}}
\end{array}
\]
The same can be done on $n_{\alpha_2}.X_0\cap X_0$ and we get
\[
\partial_{t_{1}}\partial_{t_{2}}=(y''^{2}-t''_{2})\partial_{t''_{1}}\partial_{t''_{2}}+t''_{1}\partial_{t''_{1}}^{2}-y''\partial_{t''_{1}}\partial_{y''}-x''y''\partial_{t''_{1}}\partial_{z''}
\]
Hence $\partial_{t_1}\partial_{t_2}$ is the restriction of a global differential operator, which does not lie in the image of $\mathfrak{U}(\mathfrak{g})$.\\
\indent
Let us remark that in both case, $\xi$ does not depend of the representative of $a(t_1)$ (or $a(t_1,t_2)$). We can restate these results as :
\begin{theoreme}
Let $X=\overline{G/H}$ be a symmetric space with either $G$ simple connected of adjoint type, or $G=H\times H$ with $H$ simple connected of adjoint type. Let us assume that $\tilde{\Phi}$ is of type $A_1$ or $A_2$. Then the morphism $L:\mathfrak{U}(\mathfrak{g})\to D_X$ is not surjective.
\end{theoreme}
\begin{remark}
Let ${\cal T}_X$ be the tangent sheaf on $X$. The canonical sheaf $\omega_X$ is isomorphic to ${\cal L}_{-\sum\limits_{\alpha\in\Phi_1^+}\alpha}$, and $\omega_X^{\otimes -1}$ is ample (cf. \cite{DCP} (8.4)), hence $X$ is Fano. By \cite{BiBr} (4.2), we get $H^1(X,{\cal T}_X)=0$, hence the exact sequence
\[
0\to\Gamma(X,{\cal O}_X)\to\Gamma(X,{\cal D}^1_X)\to\Gamma(X,{\cal T}_X)\to 0
\]
splits, and $\Gamma(X,{\cal D}^1_X)=\Gamma({\cal T}_X)\oplus k=\textrm{Lie}(Aut^0(X))\oplus k$. If $\tilde{\Phi}$ is of type $A_1$, $\partial_{t_1}$ lies in the image of 
\[
L':\mathfrak{U}(\textrm{Lie}(Aut^0(X)))\to D_X
\]
hence $\mathfrak{g}\subsetneq\textrm{Lie}(Aut^0(X))$. For example, we know in the case of $\textrm{PGL}_2\times\textrm{PGL}_2/\textrm{PGL}_2$ that $Aut^0(X)=\textrm{PSL}_4$, (cf. \cite{Br} (2.4.5)).\\
\indent When $\tilde{\Phi}$ is of type $A_1$, $X=\overline{G/H}$ has only two $G$-orbits (the open $G$-orbit $\Omega$ and its complement $D_1$, which is a smooth irreducible divisor). Let us recall that $\textrm{Pic}(X)$ is freely spanned by the set $\Delta_X$ of colors of $X$ (ie. irreducible $B$-stable but not $G$-stable divisors). Then we have integers $a_{D}$ such that in $\textrm{Pic}(X)$, we have
\[
[D_1]=\sum\limits_{D\in\Delta_X}a_{D}[D]
\]
When $\tilde{\Phi}$ is of type $A_1$, all these integers $a_D$ are non-negative. But by \cite{Br}(2.3.2) and (2.4.2) we know that if $X$ is a wonderful variety, it is also a $Aut^0(X)$-wonderful variety, and the complement to its open $Aut^0(X)$-orbit is a union $D_1\cup\ldots\cup D_p$, where the $D_j$ are exactly the $G$-stable smooth divisors such that there exists a negative $a_{i,D_j}$, and these $a_{i,D_j}$ have been computed in \cite{Wa} for wonderful varieties of rank 1 and 2. Hence $X$ is a flag variety for $Aut^0(X)$. (cf. \cite{BiBr})
\end{remark}

\section{Structure of $D_{Y,{\cal L}}$-module of $H^0(Y,{\cal L})$ via geometric invariant theory}
Since wonderful compactifications $Y$ of symmetric spaces with reduced root system $\tilde{\Phi}$ of type $A_1$ are actually flag variety for a bigger group, we already know what the algebra $D_X$ is. In this section, we will be interested in the case of the wonderful compactification of $\textrm{PGL}_3$, $\textrm{PGL}_3/\textrm{PSO}_3$ and of $\textrm{PGL}_6/\textrm{PSp}_6$, which are of type $A_2$. We will use a description of these compactifications as direct limits of GIT quotients of some Grassmannian $X$, as it is explained in \cite{Th2}, and this limit happens to be exactly one of these GIT quotients. We will use this description and what is known about differential operators on Grassmannians to show that in these cases, $D_Y$ is of finite type, and that the $H^0(Y,{\cal L})$ are simple as $D_{Y,{\cal L}}$-modules. For more details about GIT quotients and their variations, one can check \cite{MFK}, \cite{DH}, \cite{Th1}, \cite{Re}, and \cite{BP} for quotients by a torus.

\subsection{Thaddeus' theorem}
Let us recall that if $X$ is a projective algebraic variety acted on by a reductive group $G$, and if ${\cal L}$ is a $G$-linearized very ample line bundle, we call GIT quotient of $X$ the rational map 
\[
X\to X//G=\textrm{Proj}(\bigoplus\limits_{n\geq 0}H^0(X,{\cal L}^{\otimes n})^G)
\]
which is defined over the open set of semistable points $X^{ss}({\cal L})$. Let $\textrm{Pic}^G(X)$ denote the group of isomorphism classes of $G$-linearized invertible sheaves on $X$, let $\textrm{Pic}^G(X)_0$ be the subgroup of $\textrm{Pic}^G(X)$ of homologically trivial ${\cal L}$ with trivial $G$-linearization, and let $\textrm{NS}^G(X)=\textrm{Pic}^G(X)/\textrm{Pic}^G(X)_0$. By \cite{DH}, we can parametrize the different GIT quotients with a polytope lying in the $G$-ample cone $C^{G}(X)$ of the $G$-Neron-Severi group $\textrm{NS}^G(X)$ of $X$, that is giving to the family of non-isomorphic GIT quotients of $X$ by $G$ a structure of inverse system, giving rise to the inverse limit of all GIT-quotients of $X$ by $G$, dominating all GIT quotients, which will be denoted by $\underline{X//G}$. In \cite{Th2}, Thaddeus stated the following result :
\begin{proposition}
\item[(1)] The wonderful compactification of $\textup{PGL}_n$ is isomorphic to $\underline{X//\mathbb{C}^*}$, where $X=\textrm{Gr}_n(\mathbb{C}^n\oplus\mathbb{C}^n)$, with $\mathbb{C}^*$ acting with the weight 1 on the first $\mathbb{C}^n$ and with the weight -1 on the second one.
\item[(2)] The wonderful compactification of $\textup{PGL}_n/\textup{PSO}_n$ is isomorphic to $\underline{X//\mathbb{C}^*}$, where $X=\textrm{LaGr}_n(\mathbb{C}^n\oplus\mathbb{C}^n)$ is a Grassmannian of Lagrangian subspaces (for the standard symplectic form), with $\mathbb{C}^*$ acting with the weight 1 on the first $\mathbb{C}^n$ and with the weight -1 on the second one.
\item[(3)] If $n$ is even, the wonderful compactification of $\textup{PGL}_{n}/\textrm{PSp}_{n}$ is isomorphic to $\underline{X//\mathbb{C}^*}$, where $X=\textrm{OGr}^+_{n}(\mathbb{C}^{n}\oplus\mathbb{C}^{n})$ is the connected component containing $\mathbb{C}^{n}\oplus 0$ of the Grassmannian of orthogonal subspaces (for the quadratic form defined by $\begin{pmatrix}0&I\\I&0 \end{pmatrix}$ with $\mathbb{C}^*$ acting with the weight 1 on the first $\mathbb{C}^{n}$ and with the weight -1 on the second one.
\end{proposition}
Moreover, if ${\cal L}$ is a $\mathbb{C}^*$-linearized invertible sheaf on $X$, the stability and semistability of a point $U$ can be explicitly expressed with $\dim (U\cap(\mathbb{C}^n\oplus 0))$ and $\dim (U\cap(0\oplus\mathbb{C}^n))$. For example, $U\in X$ is semistable (resp. stable) for a very ample sheaf ${\cal L}$ with trivial $\mathbb{C}^*$-linearization if and only if both these dimensions are less (resp. strictly less) than $\frac{n}{2}$.

\subsection{Case of $\overline{\textrm{PGL}_3}$}
We will now investigate the case of the wonderful compactification of $\textrm{PGL}_3$, the two other cases being similar. In this case, we only have a few non-isomorphic GIT quotients, and it happens that $\underline{X//\mathbb{C}^*}$ is exactly the GIT quotient $X//\mathbb{C}^*$ for any very ample ${\cal L}$ with trivial $\mathbb{C}^*$-linearization.\\
\indent
Let $G=\textrm{PGL}_3$, and let $T\subset B\subset G$ be the maximal torus of diagonal matrices and the Borel subgroup of upper-triangular matrices of $G$. Let us denote by $\Phi=\Phi(G,T)$ its root system, and by $\alpha_1$ and $\alpha_2$ the two simple roots obtained by this choice of $B$. Let $V=\mathbb{C}^3$, with a basis $(e_1,e_2,e_3)$, and let $(e_1^*,e_2^*,e_3^*)$ be its dual basis. Let $\mathbb{C}^*$ act on $V\oplus V^*$ with weights 1 on $V$ and -1 on $V^*$, and let $X=\textrm{Gr}_3(V\oplus V^*)$ and $Y=\overline{G}$. Let ${\cal L}$ be a very ample sheaf with trivial $\mathbb{C}^*$-linearization, let $X^{ss}(0)$ denote the set of semi-stable points for ${\cal L}$ and let $Y(0)$ be its associated GIT quotient. Since the semistable points for this ${\cal L}$ are all stable, the morphism $\pi_0:X^{ss}(0)\to Y(0)$ is a geometric quotient, ie. points in $Y(0)$ are exactly $\mathbb{C}^*$-orbits in $X^{ss}(0)$. We know we already have a $G$-equivariant dominant morphism $d_0:Y\to Y(0)$.\\
\indent
For $M\in\textrm{GL}_3$, let $\Gamma_M=\{v\oplus Mv^*,v\in V\}$ be the graph of $M$. $\Gamma_M$ is a 3-dimensional subspace of $V\oplus V^*$, with trivial intersection with $V\oplus 0$ and $0\oplus V^*$, hence $\Gamma_M\in X^{ss}(0)$. For $t\in\mathbb{C}^*$, we have
\[
t.\Gamma_M=\{tv\oplus t^{-1}Mv^*,v\in V\}=\{v\oplus t^{-2}Mv^*,v\in V\}=\Gamma_{t^{-2}M}
\]
Hence we get a morphism
\[
\begin{tabular}{rcl}
$\pi:G$&$\to$&$Y(0)$\\
$[M]$&$\mapsto$&$\{\Gamma_{x^{-1}M},x\in\mathbb{C}^*\}$
\end{tabular}
\]
Let us recall we have an embedding $X\hookrightarrow \mathbb{P}(\Lambda^3(V\oplus V^*))=:P$. $X$ is a $\mathbb{C}^*$-stable closed subset of $P$, and $Y(0)$ can be seen as a closed subset of $P(0)=\textrm{Proj}(\mathbb{C}[\Lambda^3(V\oplus V^*)]^{\mathbb{C}^*})$. Let us denote by $\iota$ the inclusion $Y(0)\to P(0)$. Let $t=\textrm{diag}(a,b,c)\in\textrm{GL}_3$. We have
\[
\iota(\pi([t]))=\{[(xe_1\oplus x^{-1}ae_1^*)\wedge(xe_2\oplus x^{-1}be_2^*)\wedge(xe_3\oplus x^{-1}ce_3^*)],x\in\mathbb{C}^*\}
\]
and $[t]\mapsto \iota(\pi([t]))$ has trivial fibres, hence $\pi_{|T}$ is an isomorphism onto its image. Let us denote by $t_1=\frac{a}{b}$ and by $t_2=\frac{b}{c}$.\\
\indent
Let $<$ be an order such that $e_1<e_2<e_3<e_1^*<e_2^*<e_3^*$, let ${\cal B}$ be the set of subfamilies of 3 vectors $b_1<b_2<b_3$ of the canonical basis of $V\oplus V^*$, and let $B\in {\cal B}$. Let $I_B=\{i|\exists j,b_j=e_i\}$ and $I_{B}^*=\{i|\exists j,b_j=e_i^*\}$, and let $\Lambda U_{I_B}^{I_B^*}:=b_1\wedge b_2\wedge b_3$. the family $(\Lambda U_{I_B}^{I_B^*})_{B\in{\cal B}}$ forms a basis of $\Lambda^3(V\oplus V^*)$, and let $((\Lambda U_{I_B}^{I_B^*})^*)$ be its dual basis. Then $\mathbb{C}[\Lambda^3(V\oplus V^*)]$ is spanned as an algebra by the $(\Lambda U_{I_B}^{I_B^*})^*$, and since the weight of $(\Lambda U_{I_B}^{I_B^*})^*$ is given by $|I_B|-|I_B^*|$, we get that $\mathbb{C}[\Lambda^3(V\oplus V^*)]^{\mathbb{C}^*}$ is of finite type. We can show that
\begin{proposition}
In the cases of $\textup{PGL}_3$, $\textup{PGL}_3/\textup{PSO}_3$ and of $\textup{PGL}_6/\textup{PSp}_6$, $d_0$ is an isomorphism.
\end{proposition}
\begin{proof}
The aim is to show that $Y(0)$ is a wonderful compactification of $G\times G/G$, and we will use that wonderful compactifications are unique up to isomorphism. Recall we have an inclusion $X\hookrightarrow \mathbb{P}(\Lambda^3(V\oplus V^*))$. Let $G'=\textrm{SL}_6$, and $\mathfrak{g}'$ its Lie algebra. We will first decompose the $G'$-module $\Lambda^3(V\oplus V^*)$ as a sum of irreducible $G\times G$-modules. It is done by looking at highest weight vectors for $G\times G$ in $\Lambda^3(V\oplus V^*)$, which are given by $\Lambda U_{1,2}^3$ (which is of weight $(\varpi_2,\varpi'_2)$, and 1 for $\mathbb{C}^*$), $\Lambda U_{1}^{2,3}$ (which is of weight $(\varpi_1,\varpi'_1)$, and -1 for $\mathbb{C}^*$), and $\Lambda U_{1,2,3}$ and $\Lambda U^{1,2,3}$, which are both of weight $0$, and respectively 3 and -3 for $\mathbb{C}^*$. Hence as a $G\times G$-module, we have
\[
\Lambda^3(V\oplus V^*)=V_{(\varpi_2,\varpi'_2)}\oplus V_{(\varpi_1,\varpi'_1)}\oplus \mathbb{C}\oplus\mathbb{C}
\]
where $V_{(\varpi_2,\varpi'_2)}=\textrm{End}(\Lambda^2(\mathbb{C}^3))$ and $V_{(\varpi_1,\varpi'_1)}=\textrm{End}(\mathbb{C}^3)$. If we denote by $(x_i)$ a basis of $V_{(\varpi_2,\varpi'_2)}^*$, by $(y_j)$ a basis of $V_{(\varpi_1,\varpi'_1)}^*$, and respectively by $z$ and $t$ a nonzero vector in $(0\oplus 0\oplus \mathbb{C}\oplus 0)^*$ and in $(0\oplus 0\oplus 0\oplus \mathbb{C})^*$, we have
\[
\mathbb{C}[\Lambda^3(V\oplus V^*)]^{\mathbb{C}^*}=\mathbb{C}[x_iy_j,x_ix_jx_kt,y_iy_jy_kz,zt]
\]
Let 
\[
\underline{V}=(V_{(\varpi_2,\varpi'_2)}\otimes V_{(\varpi_1,\varpi'_1)})\oplus(\textrm{Sym}^3(V_{(\varpi_2,\varpi'_2)})\otimes \mathbb{C})\oplus(\textrm{Sym}^3(V_{(\varpi_1,\varpi'_1)})\otimes \mathbb{C})\oplus(\mathbb{C}\otimes\mathbb{C})
\]
Then we have a surjective morphism $\textrm{Sym}(\underline{V})\to \mathbb{C}[\Lambda^3(V\oplus V^*)]^{\mathbb{C}^*}$, and we have inclusions
\[
G\hookrightarrow Y(0) \hookrightarrow \textrm{Proj}(\mathbb{C}[\Lambda^3(V\oplus V^*)]^{\mathbb{C}^*})\hookrightarrow \mathbb{P}(\underline{V})
\]
sending 1 to a $\textrm{PGL}_3$-invariant $[v]$ such that $v$ has a nonzero part in $V_{(\varpi_2,\varpi'_2)}\otimes V_{(\varpi_1,\varpi'_1)}$, which is of weight $(\rho,\rho')$. Since $\underline{V}$ has a regular special highest weight (which is $(\rho,\rho')$), and since the $\textrm{PGL}_3$-invariants in $(\textrm{Sym}^3(V_{(\varpi_2,\varpi'_2)})\otimes \mathbb{C})\oplus(\textrm{Sym}^3(V_{(\varpi_1,\varpi'_1)})\otimes \mathbb{C})\oplus(\mathbb{C}\otimes\mathbb{C})$ are all of weight 0, we get thanks to \cite{DCP} (4.1) that $Y$ is isomorphic to the closure of the image of $G$ in $\mathbb{P}(\underline{V})$, which is $Y(0)$.
\end{proof}
Now we would like to compare $X$ and $X^{ss}(0)$. The set of unstable points $X^{us}(0)=X\setminus X^{ss}(0)$ has two connected components
\[
F_1=\{U\in X,\dim(U\cap(V\oplus 0))\geq 2\}\textrm{ and }F_2=\{U\in X,\dim(U\cap(0\oplus V^*))\geq 2\}
\]
both of them being of dimension 5 (generically, we obtain $U\in F_1$ as a direct sum of a 2-dimensional subspace $W\subset V\oplus 0$ with a 1-dimensional subspace $W'\subset (V\oplus 0)/W\oplus 50\oplus V^*)$, and the same goes for $F_2$). Since $X$ is smooth and of dimension 9, if we denote by $i$ the inclusion $X^{ss}(0)\to X$, we have ${\cal O}_X=i_*{\cal O}_{X^{ss}(0)}$. Now we will show
\begin{proposition}
In the cases of $\textup{PGL}_3$, $\textup{PGL}_3/\textup{PSO}_3$ and of $\textup{PGL}_6/\textup{PSp}_6$, $D_Y$ is a $\mathbb{C}$-algebra of finite type.
\end{proposition}
\begin{proof}
Let $\Omega$ be an affine open subset of $Y(0)$. Then $\pi_0(\Omega)$ is an affine open $\mathbb{C}^*$-invariant subset of $X^{ss}(0)$, and $\mathbb{C}[\Omega]=\mathbb{C}[\pi_0^{-1}(\Omega)]^{\mathbb{C}^*}$. By seeing differential operators on $\pi_0^{-1}(\Omega)$ as endomorphisms of $\mathbb{C}[\pi_0^{-1}(\Omega)]$, we get a restriction morphism
\[
\begin{array}{rcl}
{\cal D}_X(\pi_0^{-1}(\Omega))^{\mathbb{C}^*}&\to&{\cal D}_{Y(0)}(\Omega)\\
d&\mapsto&d_{|\mathbb{C}[\Omega]}
\end{array}
\]
Let $U\in X$, and let $v\in\Lambda^3(V\oplus V^*)$ such that $\iota(U)=[v]$. Let us write $v=w_{-3}\oplus w_{-1}\oplus w_{1}\oplus w_{3}$, where $w_i$ denotes the part of $v$ of weight $i$. Then
\[
X^{ss}(0)=\{U\in X|w_{-3}\textrm{ or }w_{-1}\neq 0, w_1\textrm{ or }w_3\neq 0\}
\]
We have
\begin{lemma}
For $U\in X^{ss}(0)$, $w_1$ and $w_{-1}$ are nonzero.
\end{lemma}
Let $(b_1,b_2,b_3)$ be a basis of $U$, and write $b_i=f_i\oplus f'_i$, with $f_i\in V$ and $f'_i\in V^*$. Let us choose the $b_i$ such that the nonzero $f_i$'s and the nonzero $f'_i$'s are linearly independent. Since $U$ is semistable, we have $|\{i|f_i\neq 0\}|\geq 2$ and $|\{i|f'_i\neq 0\}|\geq 2$, hence $w_1$ and $w_{-1}$ are nonzero.\\
\indent
Hence for all $U\in X^{ss}(0)$, its stabilizer for the $\mathbb{C}^*$-action is $\mathbb{Z}/2\mathbb{Z}$. We can cover $Y(0)$ by affine open subsets $\Omega_i$ such that the following diagram commutes
\[
\begin{array}{cccc}
\pi_0^{-1}(\Omega_i)&\simeq &\mathbb{C}^*\times\Omega_i&(t,x)\\
\downarrow&&\downarrow&\downmapsto\\
\Omega_i&\leftarrow&\mathbb{C}^*\times\Omega_i&(t^2,x)\\
x&\mapsfrom&(t,x)&
\end{array}
\]
We have isomorphisms
\[\begin{array}{rcl}
\mathbb{C}[\pi_0^{-1}(\Omega_i)]&\simeq&\mathbb{C}[\Omega_i]\otimes \mathbb{C}[t_i^{\pm 2}]\\
{\cal D}_X(\pi_0^{-1}(\Omega_i))&\simeq&\mathbb{C}[t_i^{\pm 2},\partial_{t_i^2}]\otimes{\cal D}_{Y(0)}(\Omega_i)\\
{\cal D}_X(\pi_0^{-1}(\Omega_i))^{\mathbb{C}^*}&\simeq&\mathbb{C}[t_i\partial_{t_i}]\otimes{\cal D}_{Y(0)}(\Omega_i)
\end{array}
\]
And we can lift differential operators on $\Omega_i$ with
\[
\begin{array}{rcl}
r_n : {\cal D}_{Y(0)}(\Omega_i)&\to&{\cal D}_X(\pi_0^{-1}(\Omega_i))^{\mathbb{C}^*}\\
d&\mapsto&\tilde{d}:=(t_i\partial_{t_i})^n\otimes d
\end{array}
\]
for all $n\geq 0$. Moreover, if $d\in {\cal D}_{Y(0)}(\Omega_i\cap \Omega_j)$, then $r_n(d_{|\Omega_i})=r_n(d_{|\Omega_j})$ on $\mathbb{C}[\pi_0^{-1}(\Omega_i\cap \Omega_j)]$, and $t_i\partial_{t_i}=\pm t_j\partial_{t_j}$. Hence we can glue the $r_n$'s together and get maps
\[
r_n:{\cal D}_{Y(0)}(\Omega)\to{\cal D}_X(\pi_0^{-1}(\Omega))^{\mathbb{C}^*}
\]
for all $n\geq 0$, and for each open $\Omega\subset Y(0)$ we have an isomorphism
\[
{\cal D}_Y(\pi_ 0^{-1}(\Omega))^{\mathbb{C}^*}\simeq\mathbb{C}[t\partial_{t}]\otimes{\cal D}_{Y(0)}(\Omega)
\]
Since $\textrm{codim} X^{us}(0)\geq 2$, we get
\[
D_X^{\mathbb{C}^*}={\cal D}_X(X^{ss}(0))^{\mathbb{C}^*}=\mathbb{C}[t\partial_t]\otimes D_Y
\]
\indent
Recall we have a filtration of $D_X$ by the $D_X^m$, which are the global differential operators of degree at most $m$. Since $X$ is a flag variety for $G'$, $\textrm{gr}(D_X)$ is a commutative $\mathbb{C}$-algebra of finite type (cf. \cite{BoB1}), and $\textrm{gr}(D_X)^{\mathbb{C}^*}$ is of finite type. Since $\mathbb{C}^*$ stabilizes $D_X^m$, $D_X^{m-1}$ admits a $\mathbb{C}^*$-stable supplement in $D_X^m$, which will be denoted by $E^m$. We have
\[
\begin{array}{rcl}
\textrm{gr}(D_X^{\mathbb{C}^*})&=&\bigoplus\limits_{n\geq 0}(D_X^{m})^{\mathbb{C}^*}/(D_X^{m-1})^{\mathbb{C}^*}\\
&=&\bigoplus\limits_{n\geq 0}(E^m)^{\mathbb{C}^*}\\
&=&(\bigoplus\limits_{n\geq 0}E^m)^{\mathbb{C}^*}\\
&=&(\textrm{gr}(D_X))^{\mathbb{C}^*}
\end{array}
\]
Hence $\textrm{gr}(D_{X}^{\mathbb{C}^*})$ and $D_X^{\mathbb{C}^*}$ are $\mathbb{C}$-algebras of finite type, and since from above we have a surjective morphism $D_X^{\mathbb{C}^*}\to D_Y$, $D_Y$ is also of finite type.
\end{proof}
We will now look at the $D_{X,{\cal L}}$-module structure of $H^0(X,{\cal L})$ for an invertible sheaf ${\cal L}$ on $Y$. Since $Y$ is smooth, we have a correspondence between invertible sheaves on $Y$ and Cartier divisors on $Y$. Let $(f_i)$ be the Cartier divisor associated to ${\cal L}$. Since the quotient $X^{ss}(0)\to Y(0)$ is a good categorical quotient, each $f_i$ can be seen as $\mathbb{C}^*$-invariant rational functions on $X^{ss}(0)$, which can be uniquely extended to $\mathbb{C}^*$-invariant functions $\tilde{f}_i$ on $X$ such that the $\frac{f_i}{f_j}$ are regular when needed. Let us denote by $\overline{\cal L}$ the invertible sheaf on $X$ associated to the Cartier divisor $(\tilde{f}_i)$. We have for each open subset $\Omega\subset Y(0)$
\[
{\cal L}(\Omega)\simeq\overline{\cal L}(\pi_0^{-1}(\Omega))^{\mathbb{C}^*}
\]
If ${\cal L}={\cal O}_{Y}(d)$ for a Weil divisor $d$, then $\overline{\cal L}={\cal O}_X(D)$ with $D:=\overline{\pi_0^{-1}(d)}$. To avoid confusion, we will now use ${\cal L}$ to denote invertible sheaves on $Y$, and $\overline{\cal L}$ to denote invertible sheaves on $X$.\\
\indent
Let
\[
T'=\Big\{
\begin{pmatrix}
*&0&0&0&0&0\\0&*&0&0&0&0\\0&0&*&0&0&0\\0&0&0&*&0&0\\0&0&0&0&*&0\\0&0&0&0&0&*
\end{pmatrix}\Big\}, 
B'=\Big\{
\begin{pmatrix}
*&*&*&*&*&*\\0&*&*&*&*&*\\0&0&*&*&*&*\\0&0&0&*&*&*\\0&0&0&0&*&*\\0&0&0&0&0&*
\end{pmatrix}\Big\}, 
P'=\Big\{
\begin{pmatrix}
*&*&*&*&*&*\\ *&*&*&*&*&*\\ *&*&*&*&*&*\\0&0&0&*&*&*\\0&0&0&*&*&*\\0&0&0&*&*&*
\end{pmatrix}\Big\}
\]
be subgroups of $G'$. Let $\beta_1,\ldots,\beta_5$ be the simple roots, and $\omega_1,\ldots,\omega_5$ the associated fundamental weights. We have $X=G'/P'$, and $P'$ is the parabolic maximal subgroup associated to the simple root $\beta_3$. Since $\textrm{Pic}(X)=\mathbb{Z}$, we know that invertible sheaves on $X$ are isomorphic to some $\overline{\cal L}_{k\varpi_3}$. Moreover, $\mathbb{C}^*$ acts on $X$ via the one--parameter subgroup $2\varpi_3^{\vee}$.\\
\indent
Let us recall that $\textrm{Pic}(Y)$ is spanned by ${\cal O}_Y(d_1)$ and ${\cal O}_Y(d_2)$, where $d_1$ and $d_2$ are the colors (irreducible $B$-stable but not $G$-stables divisors) of $X$. Let $D_i=\overline{\pi_0^{-1}(d_i)}$. There exists $k_1$ and $k_2$ such that ${\cal O}_X(D_i)=\overline{\cal L}_{k_i\varpi_3}$. The $B\times B^-$ open orbit of $G$ is given by the equations
\[
g_{3,3}\neq 0 \textrm{ and }g_{2,2}g_{3,3}-g_{2,3}g_{3,2}\neq 0
\]
where $g_{ij}$ denotes the $(i,j)$-coefficient of $g\in\textrm{GL}_3$. Up to renumbering the $d_i$'s, the $D_i$'s are given on $\pi_0^{-1}(G)$ by the equations
\[
F_1:=\frac{(\Lambda U_{1,2}^3)^*}{(\Lambda U_{1,2,3})^*}=0\textrm{ and }F_2:=\frac{(\Lambda U_1^{2,3})^*}{(\Lambda U_{1,2,3})^*}=0
\]
which are respectively of weight -2 and -4 for the action of $\mathbb{C}^*$. Let $U_i$ be a trivializing open subset for $\overline{\cal L}_{k_i\varpi_3}$, and let $\Omega_i$ be its intersection with $\pi_0^{-1}(G)$. We have
\[
\begin{array}{rcl}
{\cal O}_X(D_i)(\Omega_i)^{\mathbb{C}^*}&=&{\cal O}_X(\Omega_i)_{-2^i}.\frac{1}{F_i}\\
&=&\overline{\cal L}_{k_i\varpi_3}(\Omega_i)_{3k_i-2^i}.\frac{1}{F_i\sigma_{k_i\varpi_3}}
\end{array}
\]
where $\sigma_{k_i\varpi_3}$ is a trivializing section of $\overline{\cal L}_{k_i\varpi_3}$, and where $\overline{\cal L}(\Omega_i)_n$ denotes the weight $n$ part of $\overline{\cal L}(\Omega_i)$. Hence for each $\mathbb{C}^*$-invariant open subset $U\subset X$, we have
\[
{\cal O}_X(D_i)^{\mathbb{C}^*}\simeq\overline{\cal L}_{k_i\varpi_3}(\Omega_i)_{3k_i-2^i}
\]
and for all invertible sheaf ${\cal L}$ on $Y$, there exists $k,n\in\mathbb{Z}$ such that for all open $U\subset Y$ we have
\[
{\cal L}(U)\simeq\overline{\cal L}_{k\varpi_3}(\pi_0^{-1}(U))_n
\]
Recall that if ${\cal L}$ is an invertible sheaf on $Y$, we can define the sheaf of differential operators of ${\cal L}$ as 
\[
{\cal D}_{Y,{\cal L}}={\cal L}\otimes{\cal D}_Y\otimes{\cal L}^{\otimes -1}
\]
and let $D_{Y,{\cal L}}$ be the algebra of its global section. We can show
\begin{theoreme}
In the cases of $\textup{PGL}_3$, $\textup{PGL}_3/\textup{PSO}_3$ and of $\textup{PGL}_6/\textup{PSp}_6$, let ${\cal L}$ be a invertible sheaf on $Y$. Then $H^0(Y,{\cal L})$ is either 0 or simple as a left $D_{Y,{\cal L}}$-module.
\end{theoreme}
\begin{proof}
Let us assume $H^0(Y,{\cal L})\neq 0$. Let $k,n$ be integers such that for all open $U\subset Y$, we have ${\cal L}(U)\simeq\overline{\cal L}_{k\varpi_3}(\pi_0^{-1}(U))_n$. We have $H^0(Y,{\cal L})=H^0(X,\overline{\cal L}_{k\varpi_3})_n$, and since $X$ is a flag variety, $H^0(X,\overline{\cal L}_{k\varpi_3})$ is a simple left $D_{X,k\varpi_3}$-module, and $H^0(X,\overline{\cal L}_{k\varpi_3})_n$ is a simple left $D_{X,k\varpi_3}^{\mathbb{C}^*}$-module.\\
\indent
Let $\Omega\subset Y$ be an affine open subset. We write
\[
{\cal D}_{X,k\varpi_3}(\pi_0^{-1}(\Omega))^{\mathbb{C}^*}=\bigoplus\limits_{p,q\in\mathbb{Z}}\overline{\cal L}_{k\varpi_3}(\pi_0^{-1}(\Omega))_{n+2p}\otimes{\cal D}_X(\pi_0^{-1}(\Omega))_{2q}\otimes\overline{\cal L}^{\otimes -1}_{k\varpi_3}(\pi_0^{-1}(\Omega))_{-n-2p-2q}
\]
A section of $((\overline{\cal L}_{k\varpi_3})_{n+2p}\otimes({\cal D}_X)_{2q}\otimes(\overline{\cal L}^{\otimes -1}_{k\varpi_3})_{-n-2p-2q})(\pi_0^{-1}(\Omega))$ can be written as a linear combination of
\[
s=t^{2p}\sigma\otimes t^{2q}\partial(t^2\partial_{t^2})^k\otimes t^{-2p-2q}\tau
\]
where $\sigma\in\overline{\cal L}_{k\varpi_3}(\pi_0^{-1}(\Omega))_n$, $\partial\in{\cal D}_Y(\Omega)$, $\tau\in\overline{\cal L}^{\otimes -1}_{k\varpi_3}(\pi_0^{-1}(\Omega))_{-n}$ and $k\geq 0$. $s$ acts on $\overline{\cal L}_{k\varpi_3}(\pi_0^{-1}(\Omega))_n$ as $c(p,q,k,n).\sigma\otimes\partial\otimes\tau$, where $c(p,q,k,n)$ is a scalar which does not depend on $\Omega$ ($\partial$ commutes with $t^2$). Hence $H^0(X,\overline{\cal L})_n$ is a simple left $H^0(X,(\overline{\cal L}_{k\varpi_3})_{n}\otimes{\cal D}_Y \otimes(\overline{\cal L}^{\otimes -1}_{k\varpi_3})_{-n})$-module, and since
\[
H^0(X,(\overline{\cal L}_{k\varpi_3})_{n}\otimes {\cal D}_Y \otimes(\overline{\cal L}^{\otimes -1}_{k\varpi_3})_{-n})=D_{Y,{\cal L}}
\]
it is a simple left $D_{Y,{\cal L}}$-module.
\end{proof}
\begin{remark}
When $\tilde{G}$ is a connected adjoint group of type $E_6$, there is an order two outer automorphism of $E_6$ whose fixed points form a group $H$ of type $F_4$, and the restricted root system of the symmetric space $G/H$ is of type $A_2$. In this case, we still can show that $Y=\overline{\tilde{G}/H}$ is isomorphic to a GIT quotient by a $\mathbb{C}^*$ of a flag variety $G'/P$, where $G'$ is simply connected of type $E_7$ and where $P$ is a maximal parabolic, and that $D_Y$ is of finite type. This is done using the following description of groups of type $E_6$ and $E_7$, which can be found in \cite{Sp}.\\
\indent
Let $A:=V_{\varpi_1}$ and $B:=V_{\varpi_6}$ be the two non-isomorphic irreducible representations of dimension 27 of a group of type $E_6$. Then we can construct a connected simply connected group $G$ of type $E_6$ as a subgroup of $\textrm{GL}(A)\times \textrm{GL}(B)$, and $H$ can be seen as the fixed points of the automorphism permuting $A$ and $B$. Moreover, as vector spaces, $A\simeq B\simeq M_{3}\oplus M_3\oplus M_3$, where $M_3$ is the space of $3\times 3$ matrices. Define 
\[
\begin{array}{rcl}
i:\mathbb{C}^*&\to &\textrm{GL}(A)\times \textrm{GL}(B)\\
t&\mapsto&(t,t^{-1})
\end{array}
\]
and $H_1=i(\mathbb{C}^*).G$.\\
\indent
Now let $V_{\varpi_7}$ be the irreducible representation of highest weight $\varpi_7$ of $G'$. Its dimension is 56, and it can be decomposed as a $G$-module as
\[
V_{\varpi_7}=A\oplus B\oplus\mathbb{C}\oplus\mathbb{C}
\]
Moreover, we can see $H_1$ as the subgroup of $G'$ stabilizing this decomposition. Let $w=I_3\oplus I_3\oplus I_3\in M_{3}\oplus M_3\oplus M_3$, and let $v=(w,w,1,1)\in V_{\varpi_7}$. Then the map
\[
\begin{array}{rcl}
H_1&\to&\mathbb{P}(V_{\varpi_7})\\
h&\mapsto&h.[v]
\end{array}
\]
has fibres isomorphic to $H$, which gives an inclusion $H_1/H\hookrightarrow \mathbb{P}(V_{\varpi_7})$, and this map factors through $G'/P$, where $P$ is a conjugate of $P_{\varpi_7}$. Hence we have $H_1/H\hookrightarrow G'/P \hookrightarrow \mathbb{P}(V_{\varpi_7})$. Moreover, $i$ is inducing an action of $\mathbb{C}^*$ on $V_{\varpi_7}$, with weights 1, -1, -3 and 3 on $A$, $B$, $0\oplus 0\oplus \mathbb{C}\oplus 0$ and $0\oplus 0\oplus 0\oplus \mathbb{C}$, and we have morphisms
\[
\tilde{G}/H\hookrightarrow (G'/P)//\mathbb{C}^*\hookrightarrow \mathbb{P}(V_{\varpi_7})//\mathbb{C}^*
\]
where quotients are done for ample line bundles with trivial $\mathbb{C}^*$-linearization. If we denote by $(x_i)$ a basis of $A^*$, $(y_j)$ a basis of $B^*$, and by $z$ and $t$ two nonzero vectors respectively of $(0\oplus 0\oplus \mathbb{C}\oplus 0)^*$ and $(0\oplus 0\oplus 0\oplus \mathbb{C})^*$, we have
\[
\mathbb{C}[V_{\varpi_7}]^{\mathbb{C}^*}=\mathbb{C}[x_iy_j,x_ix_jx_kt,y_iy_jy_kt,zt]
\]
Let 
\[
\underline{V}=(A\otimes B)\oplus(\textrm{Sym}^3(A)\otimes \mathbb{C})\oplus(\textrm{Sym}^3(B)\otimes \mathbb{C})\oplus(\mathbb{C}\otimes\mathbb{C})
\]
Then we have a surjective morphism $\textrm{Sym}(\underline{V})\to \mathbb{C}[V_{\varpi_7}]^{\mathbb{C}^*}$, and we have inclusions
\[
\tilde{G}/H \hookrightarrow (G'/P)//\mathbb{C}^* \hookrightarrow \textrm{Proj}(\mathbb{C}[V_{\varpi_7}]^{\mathbb{C}^*})\hookrightarrow \mathbb{P}(\underline{V})
\]
Since $\underline{V}$ has a regular special highest weight (which is $\varpi_1+\varpi_6$), and since the $H$-invariants in $(\textrm{Sym}^3(A)\otimes \mathbb{C})\oplus(\textrm{Sym}^3(B)\otimes \mathbb{C})\oplus(\mathbb{C}\otimes\mathbb{C})$ have weight 0, $Y$ is isomorphic to the closure of the image of $\tilde{G}/H$ in $\mathbb{P}(\underline{V})$, which is $(G'/P)//\mathbb{C}^*$.\\
\indent
Using the description in \cite{BP} of (semi-)stable points for GIT quotients by a torus, we still have that
\[
X^{s}(0)=X^{ss}(0)=\{U\in X|w_{-3}\textrm{ or }w_{-1}\neq 0, w_1\textrm{ or }w_3\neq 0\}
\]
and that the quotient $(G'/P)//\mathbb{C}^*$ is geometric. Hence the same proof than before gives us that $D_Y$ is of finite type, since $G'/P$ is a flag variety for $G'$.
\end{remark}

\section{$D_{Y,{\cal L}}$-module structure of higher cohomology groups}
Since $X^{us}(0)$ is not empty, it is not necessarily true when $i>0$ that $H^i(Y,{\cal L})=H^i(X,\overline{\cal L}_{k\varpi_3})_n$ for some $k$ and $n$. The difference can be expressed in terms of local cohomology groups $H^i_{X^{us}}(X,\overline{\cal L}_{k\varpi_3})$ that we will now introduce. One can check \cite{SGA2}, \cite{Ha}, \cite{Ke1} for more details.
\subsection{Local cohomology}
Let $X$ be an algebraic variety, let $Z\subset X$ be a closed subset, and let ${\cal F}$ be a ${\cal O}_X$-module. Let 
\[
\Gamma_Z(X,{\cal F})=\{\sigma\in\Gamma(X,{\cal F})|\sigma_{|X\setminus Z}=0\}
\]
If $Z_2\subset Z_1$ are two closed subsets of $X$, we have an injection $\Gamma_{Z_2}(X,{\cal F})\to \Gamma_{Z_1}(X,{\cal F})$, and let
\[
\Gamma_{Z_1/Z_2}(X,{\cal F})=\Gamma_{Z_1}(X,{\cal F})/\Gamma_{Z_2}(X,{\cal F})
\]
Let $H^i_{Z}(X,\bullet)$ and $H^i_{Z_1/Z_2}(X,\bullet)$ be the right derived functors of $\Gamma_{Z}(X,\bullet)$ and $\Gamma_{Z_1/Z_2}(X,\bullet)$. This is done using Godement resolutions for $H^i_{Z_1/Z_2}(X,\bullet)$ since $\Gamma_{Z_1/Z_2}(X,\bullet)$ is not necessarily left exact. Remark we have $H^i_{Z/\emptyset}(X,{\cal F})=H^i_{Z}(X,{\cal F})$. Now let $\underline{\Gamma}_{Z_1/Z_2}({\cal F})$ be the sheaf associated to the presheaf $U\mapsto \Gamma_{(Z_1\cap U)/(Z_2\cap U)}(U,{\cal F}_{|U})$, and ${\cal H}^i_{Z_1/Z_2}({\cal F})$ be the sheaf associated to the presheaf $U\mapsto H_{(Z_1\cap U)/(Z_2\cap U)}(U,{\cal F}_{|U})$. We will recall a few properties of these functors.
\begin{proposition} Let ${\cal F}$ be a ${\cal O}_X$-module, and let $Z_3\subset Z_2\subset Z_1$ be three closed subsets of $X$.
\item[(1)] We have two long exact sequences
\[
\ldots\to H^{i-1}_{Z_1/Z_2}(X,{\cal F})\to H^i_{Z_2/Z_3}(X,{\cal F})\to H^i_{Z_1/Z_3}(X,{\cal F})\to H^i_{Z_1/Z_2}(X,{\cal F})\to H^{i+1}_{Z_2/Z_3}(X,{\cal F})\to\ldots
\]
and
\[
\ldots\to {\cal H}^{i-1}_{Z_1/Z_2}({\cal F})\to {\cal H}^i_{Z_2/Z_3}({\cal F})\to {\cal H}^i_{Z_1/Z_3}({\cal F})\to {\cal H}^i_{Z_1/Z_2}({\cal F})\to {\cal H}^{i+1}_{Z_2/Z_3}({\cal F})\to\ldots
\]
In particular, if $Z_1=X$ and $Z_3=\emptyset$, and if $j:X\setminus Z_2\to X$ is the inclusion, we have a short exact sequence
\[
0\to {\cal H}^0_{Z_2}({\cal F})\to {\cal F}\to j_*({\cal F}_{|X\setminus Z_2})\to {\cal H}^{1}_{Z_2}({\cal F})\to 0
\]
and isomorphisms ${\cal H}^{i+1}_{Z_2}({\cal F})=R^{i}j_*({\cal F}_{|X\setminus Z_2})$ for $i>0$.
\item[(2)] We have an isomorphism
\[
H^i_{Z_1/Z_2}(X,{\cal F})\simeq H^i_{Z_1\setminus Z_2}(X\setminus Z_2,{\cal F})
\]
\item[(3)] (Excision lemma) If $V\subset X$ is open and containing $Z_1$, we have
\[
H^i_{Z_1/Z_2}(X,{\cal F})\to H^i_{Z_1/Z_2}(V,{\cal F}_{|V})
\]
\item[(4)] Let $X'$ be an algebraic variety, let $f:X\to X'$ be a morphism, and $Z'_2\subset Z'_1$ two closed subsets of $X'$. We have a spectral sequence
\[
H^p_{Z_1/Z_2}(X',R^qf_*{\cal F})\Rightarrow H^*_{f^{-1}Z_1/f^{-1}Z_2}(X,{\cal F})
\]
\item[(5)] If $j:X\setminus Z_2\to X$ is the inclusion, we have a spectral sequence
\[
R^pj_*{\cal H}^q_{Z_1\setminus Z_2}({\cal F}_{|X\setminus Z_2})\Rightarrow {\cal H}^*_{Z_1/Z_2}({\cal F})
\]
\item[(6)] If $Z_1$ is the disjoint union of two closed subsets $A$ and $B$, we have
\[
H^i_{Z_1}(X,{\cal F})=H^i_{A}(X,{\cal F})\oplus H^i_{B}(X,{\cal F})
\]
\item[(7)] If $W_2\subset W_1$ are two closed subsets of $X$, let $S_1:=Z_1\cap W_1$, and let $S_2:=(W_1\cap Z_2)\cup(W_2\cap Z_1)$. We have a spectral sequence
\[
H^p_{W_1/W_2}(X,{\cal H}^q_{Z_1/Z_2}({\cal F}))\Rightarrow H^*_{S_1/S_2}(X,{\cal F})
\]
In particular, when $W_1=X$ and $W_2=\emptyset$, we get the spectral sequence
\[
H^p(X,{\cal H}^q_{Z_1/Z_2}({\cal F}))\Rightarrow H^*_{Z_1/Z_2}(X,{\cal F}) \indent (*)
\]
\end{proposition}
Let us show a sufficient condition for the last spectral sequence to degenerate.
\begin{proposition}
If $Z_1\setminus Z_2$ is affine, and is ${\cal F}$ is quasi-coherent, we have
\[
H^i_{Z_1/Z_2}(X,{\cal F})=H^0(X,{\cal H}^i_{Z_1/Z_2}(X,{\cal F}))
\]
\end{proposition}
\begin{proof}
Let us recall the following properties :\\
\textbullet \indent If $Y$ is closed and ${\cal F}$ quasi-coherent, then ${\cal H}^i_Y({\cal F})$ is quasi-coherent for all $i$.\\
\textbullet \indent Under the same hypotheses, we have an isomorphism
\[
\lim\limits_{\to}{\cal E}xt^i_{{\cal O}_X}({\cal O}_X/{\cal I}_Y^{\otimes m},{\cal F})\to{\cal H}^i_{Y}({\cal F})
\]
\textbullet \indent If $({\cal F}_{\alpha})$ is a inverse system of abelian sheaves on $X$, we have an isomorphism
\[
H^i(X,\lim\limits_{\to}{\cal F}_{\alpha})=\lim\limits_{\to}H^i(X,{\cal F}_{\alpha})
\]
Proofs can be found respectively in \cite{SGA2} II.3, \cite{SGA2} II.6 and \cite{Ke2} 8.\\
\indent
Let ${\cal F}'={\cal F}_{|X\setminus Z_2}$. Since $Z_1\setminus Z_2$ is affine, the map
\[
{\cal E}xt^i_{{\cal O}_{X\setminus Z_2}}({\cal O}_{X\setminus Z_2}/{\cal I}_{Z_1\setminus Z_2}^{\otimes m},{\cal F}')\to {\cal E}xt^i_{{\cal O}_{X\setminus Z_2}}({\cal O}_{X\setminus Z_2}/{\cal I}_{Z_1\setminus Z_2}^{\otimes m+1},{\cal F}')
\]
is an inclusion, and it gives us a long exact sequence
\[
\begin{aligned}
\ldots&\to H^p(X\setminus Z_2,{\cal E}xt^i_{{\cal O}_{X\setminus Z_2}}({\cal O}_{X\setminus Z_2}/{\cal I}_{Z_1\setminus Z_2}^{\otimes m},{\cal F}'))\\
&\to H^p(X\setminus Z_2,{\cal E}xt^i_{{\cal O}_{X\setminus Z_2}}({\cal O}_{X\setminus Z_2}/{\cal I}_{Z_1\setminus Z_2}^{\otimes m+1},{\cal F}'))\to H^p(X\setminus Z_2,{\cal G}_m)\to\ldots
\end{aligned}
\]
where ${\cal G}_m$ denotes the quotient, which is quasi-coherent and annihilated by ${\cal I}_{Z_1\setminus Z_2}$, hence it can be seen as a ${\cal O}_{Z_1\setminus Z_2}$-module, and if $j:Z_1\setminus Z_2\to X\setminus Z_2$ is the inclusion, we have $j_*j^*{\cal G}_m={\cal G}_m$, and the spectral sequence
\[
H^p(X\setminus Z_2,R^qj_*j^*{\cal G}_m)\Rightarrow H^*(Z_1\setminus Z_2,j^*{\cal G}_m)
\]
degenerates, hence $H^p(Z_1\setminus Z_2,j^*{\cal G}_m)=H^p(X\setminus Z_2,{\cal G}_m)=0$ since $Z_1\setminus Z_2$ is affine and $j^*{\cal G}_m$ is quasi-coherent. Hence for $i>0$ we have surjections
\[
H^p(X\setminus Z_2,{\cal E}xt^i_{{\cal O}_{X\setminus Z_2}}({\cal O}_{X\setminus Z_2}/{\cal I}_{Z_1\setminus Z_2}^{\otimes m},{\cal F}'))\to H^p(X\setminus Z_2,{\cal E}xt^i_{{\cal O}_{X\setminus Z_2}}({\cal O}_{X\setminus Z_2}/{\cal I}_{Z_1\setminus Z_2}^{\otimes m+1},{\cal F}'))
\]
that give us $H^p(X\setminus Z_2,{\cal E}xt^i_{{\cal O}_{X\setminus Z_2}}({\cal O}_{X\setminus Z_2}/{\cal I}_{Z_1\setminus Z_2}^{\otimes m},{\cal F}'))=0$, since it is true for $m=0$. Then for $i>0$ we have
\[
\begin{aligned}
&H^p(X\setminus Z_2,{\cal H}^q_{Z_1\setminus Z_2}({\cal F}'))\\
=&H^p(X\setminus Z_2,\lim\limits_{\to}{\cal E}xt^i_{{\cal O}_{X\setminus Z_2}}({\cal O}_{X\setminus Z_2}/{\cal I}_{Z_1\setminus Z_2}^{\otimes m},{\cal F}'))\\
=&\lim\limits_{\to}H^p(X\setminus Z_2,{\cal E}xt^i_{{\cal O}_{X\setminus Z_2}}({\cal O}_{X\setminus Z_2}/{\cal I}_{Z_1\setminus Z_2}^{\otimes m},{\cal F}'))=0
\end{aligned}
\]
Hence
\[
H^i_{Z_1\setminus Z_2}(X\setminus Z_2,{\cal F}')=H^0(X\setminus Z_2,{\cal H}^i_{Z_1\setminus Z_2}({\cal F}'))
\]
for $i>0$. If $\iota:X\setminus Z_2\to X$ is the inclusion, for the same reasons as before, we have $R^p\iota_*{\cal H}^q_{Z_1\setminus Z_2}({\cal F}')$ for all $p>0$, hence the spectral sequence
\[
R^p\iota_*{\cal H}^q_{Z_1\setminus Z_2}({\cal F}')\Rightarrow {\cal H}^*_{Z_1/Z_2}({\cal F})
\]
degenerates, and we get
\[
\begin{aligned}
H^i_{Z_1/Z_2}(X,{\cal F})&=H^i_{Z_1\setminus Z_2}(X\setminus Z_2,{\cal F}')\\
&=H^0(X\setminus Z_2,{\cal H}^i_{Z_1\setminus Z_2}({\cal F}'))\\
&=H^0(X,\iota_*{\cal H}^i_{Z_1\setminus Z_2}({\cal F}'))\\
&=H^0(X,{\cal H}^i_{Z_1/Z_2}({\cal F}))
\end{aligned}
\]
\end{proof}
Thanks to the excision lemma and to the spectral sequence $(*)$, it is easy to show that
\begin{proposition} Let $X$ be a Cohen-Macaulay variety (ie. all local rings ${\cal O}_{X,x}$ have dimension equal to their depth), and ${\cal F}$ a locally free sheaf on $X$.
\item[(1)] If $Z\subset X$ is closed, then ${\cal H}^i_Z({\cal F})=0$ for $i<\textup{codim}(Z)$. Moreover, if ${\cal I}_Z$ is locally spanned by $\textup{codim}(Z)$ elements, then ${\cal H}^i_Z({\cal F})=0$ for $i>\textup{codim}(Z)$.
\item[(2)] If $Z\subset X$ is locally closed, then $H^i_Z(X,{\cal F})=0$ for $i<\textup{codim}(Z)$. Moreover, if $V$ is open and containing $Z$ as a closed subset, if $Z$ is affine, and if ${\cal I}_Z$ is locally spanned by $\textup{codim}(Z)$ elements, then $H^i_Z(X,{\cal F})=0$ for $i>\textup{codim}(Z)$.
\end{proposition}
\subsection{Cousin complexes}
If $Z_4\subset Z_3\subset Z_2\subset Z_1$ are four closed subsets of $X$, we have different long exact sequences that give us the following diagram
\begin{center}
\begin{changemargin}{-1cm}{0cm}
\begin{tikzcd}
{}&&&\vdots \arrow{d}	&\udots \arrow{ld} & \\
	&	&	&H^{j-1}_{Z_1/Z_2}(X,{\cal F}) \arrow{ld} \arrow{d}	& & \\
	&\ldots \arrow{r}	&H^j_{Z_2/Z_4}(X,{\cal F}) \arrow{ld} \arrow{r}
								&H^j_{Z_2/Z_3}(X,{\cal F}) \arrow{d} \arrow{r}
												&H^{j+1}_{Z_3/Z_4}(X,{\cal F}) \arrow{r} &\ldots \\
	&H^j_{Z_1/Z_4}(X,{\cal F}) \arrow{ld}	&	&H^j_{Z_1/Z_3}(X,{\cal F}) \arrow{d}	& & \\
\udots &		&	&\vdots	& &\\
\end{tikzcd}
\end{changemargin}
\end{center}
Hence we have morphisms
\[
H^{j-1}_{Z_1/Z_2}(X,{\cal F})\overset{d}\to H^j_{Z_2/Z_3}(X,{\cal F})\overset{d'}\to H^{j+1}_{Z_3/Z_4}(X,{\cal F})
\]
The idea of Cousin complexes is to generalize this to a filtration 
\[
\{Z\}=(\emptyset=Z_{n+1}\subset Z_n\subset\ldots\subset Z_1\subset Z_0=X)
\]
of $X$. We obtain complexes
\[
0\to H^0(X,{\cal F})\to H^0_{Z_0/Z_1}(X,{\cal F})\to H^1_{Z_1/Z_2}(X,{\cal F})\to \ldots \to H^{n}_{Z_n}(X,{\cal F})\to 0
\]
and
\[
0\to {\cal F}\to {\cal H}^0_{Z_0/Z_1}({\cal F})\to {\cal H}^1_{Z_1/Z_2}({\cal F})\to \ldots \to {\cal H}^{n}_{Z_n}({\cal F})\to 0
\]
which will be respectively denoted by $\textup{Cousin}_{\{Z\}}{\cal F}$ and $\underline{{\cal C}\textup{ousin}}_{\{Z\}}{\cal F}$. In \cite{Ke1}, Kempf showed that
\begin{proposition}
Assume that
\item[(a)] the local Cousin complex $\underline{{\cal C}\textup{ousin}}_{\{Z\}}{\cal F}$ is exact ;
\item[(b)] for all $i$, $Z_i\setminus Z_{i+1}$ is affine.\\
\indent
Then :
\item[(1)] $\textup{Cousin}_{\{Z\}}{\cal F}$ is the complex of global sections of $\underline{{\cal C}\textup{ousin}}_{\{Z\}}{\cal F}$ ;
\item[(2)] the $i$-th homology group of $\textup{Cousin}_{\{Z\}}{\cal F}$ is isomorphic to $H^i(X,{\cal F})$.
In particular, if $H^i(X,{\cal F})=0$ for all $i>0$, the complex $\textup{Cousin}_{\{Z\}}{\cal F}$ is a resolution of $H^0(X,{\cal F})$.
\end{proposition}
Moreover, he gave the following equivalent conditions to (a) :
\begin{itemize}
\item[(A)] ${\cal H}^i_{Z_p}({\cal F})=0$ when $p\neq i$ ;
\item[(A1)] ${\cal H}^i_{Z_p}({\cal F})=0$ when $p>i$ ;
\item[(A2)] ${\cal H}^i_{Z_p/Z_{p+1}}({\cal F})=0$ when $p\neq i$ 
\end{itemize}
and when $\{Z\}$ is a filtration of $X$ satisfying (b) and such that for all $i$, $\textrm{codim}(Z_i)=i$, (A) is satisfied for all locally free sheaves.\\
\indent
Now assume we have a filtration $\{Z\}$ of $X$ satisfying $(b)$ such that for all $i$, $\textrm{codim}(Z_i)=i$, and let ${\cal F}$ be a quasi-coherent sheaf on $X$.
\begin{definition}
We define the $i$-th restricted Cousin complex ${\cal F}$ relatively to $\{Z\}$ as the complex
\[
0\to H^i_{Z_i}(X,{\cal F})\to H^i_{Z_i/Z_{i+1}}(X,{\cal F})\to H^{i+1}_{Z_{i+1}/Z_{i+2}}(X,{\cal F})\to \ldots\to H^n_{Z_n}(X,{\cal F})\to 0
\]
and it will be denoted as $\textup{Cousin}_{\{Z\},i}{\cal F}$. Equivalently, we define the $i$-th restricted local Cousin complex of ${\cal F}$ relatively to $\{Z\}$ as
\[
0\to {\cal H}^i_{Z_i}({\cal F})\to {\cal H}^i_{Z_i/Z_{i+1}}({\cal F})\to {\cal H}^{i+1}_{Z_{i+1}/Z_{i+2}}({\cal F})\to \ldots\to {\cal H}^n_{Z_n}({\cal F})\to 0
\]
and it will be denoted as $\underline{{\cal C}\textup{ousin}}_{\{Z\},i}{\cal F}$
\end{definition}
They are actually complexes since the diagram
\begin{center}
\begin{tikzcd}
{}				&											& H^i_{Z_i}(X,{\cal F}) \arrow{d} \arrow{ld}	&													&		\\
\ldots \arrow{r} & H^i_{Z_i/Z_{i+2}}(X,{\cal F}) \arrow{r}	& H^i_{Z_i/Z_{i+1}}(X,{\cal F}) \arrow{r}	& H^{i+1}_{Z_{i+1}/Z_{i+2}}(X,{\cal F}) \arrow{r}	&\ldots
\end{tikzcd}
\end{center}
commutes, and since bottom line is exact. The $i$-th restricted Cousin complex allows us to compute the $H^j(X,{\cal H}^i_{Z_i}{\cal F})$ :
\begin{theoreme}\label{cousinrestreintlocalglobal}
Let $\{Z\}$ be a filtration of $X$ satisfying $(C)$, and let ${\cal F}$ be a quasi-coherent sheaf satisfying $(a)$. Then
\item[(1)] we have an isomorphism $\textup{Cousin}_{\{Z\},i}{\cal F}\to H^0(X,\underline{{\cal C}\textup{ousin}}_{\{Z\},i}{\cal F})$ ;
\item[(2)] le $j$-th homology group of $\textup{Cousin}_{\{Z\},i}{\cal F}$ is isomorphic to $H^j(X,{\cal H}^i_{Z_i}{\cal F})$. In particular, when the $H^j(X,{\cal H}^i_{Z_i}{\cal F})$ for $j>0$ are all zero, the complex $\textup{Cousin}_{\{Z\},i}{\cal F}$ is a resolution of $H^i(X,{\cal F})$.
\end{theoreme}
\begin{proof}
Let us recall that for all $Z_2\subset Z_1$, $W_2\subset W_1$, and $S_1:=Z_1\cap W_1$, $S_2=(W_1\cap Z_2)\cup(W_2\cap Z_1)$, we have a spectral sequence
\[
H^p_{W_1/W_2}(X,{\cal H}^q_{Z_1/Z_2}({\cal F}))\Rightarrow H^*_{S_1/S_2}(X,{\cal F})
\]
Let ${\cal F}$ satisfying the equivalent conditions (a),(A),(A1) and (A2).\\
\indent
Let $k\geq i$. Since (A2), we have ${\cal H}^{j}_{Z_{k}/Z_{k+1}}({\cal F})=0$ for $j\neq k$. The spectral sequence
\[
H^p_{Z_i}(X,{\cal H}^q_{Z_k/Z_{k+1}}({\cal F}))\Rightarrow H^*_{Z_k/Z_{k+1}}(X,{\cal F})
\]
degenerates and gives us $H^k_{Z_k/Z_{k+1}}(X,{\cal F})=H^0_{Z_i}(X,{\cal H}^k_{Z_k/Z_{k+1}}({\cal F}))$. Moreover, (A1) gives us ${\cal H}^p_{Z_i}({\cal F})=0$ for $p<i$, hence we get with the spectral sequence
\[
H^p_{Z_i}(X,{\cal H}^q_{Z_i}({\cal F}))\Rightarrow H^*_{Z_i}(X,{\cal F})
\]
an isomorphism $H^i_{Z_i}(X,{\cal F})=H^0_{Z_i}(X,{\cal H}^i_{Z_i}({\cal F}))$, which proves (1).\\
\indent
To prove (2), let us start to show that $\underline{{\cal C}\textup{ousin}}_{\{Z\},i}{\cal F}$ is a resolution of ${\cal H}^i_{Z_i}({\cal F})$. Let $p\geq i$. Since ${\cal F}$ satisfies (A), we have
\[
{\cal H}^p_{Z_{p+1}}({\cal F})={\cal H}^{p+1}_{Z_p}({\cal F})=0
\]
and we obtain a short exact sequence
\[
0\to {\cal H}^p_{Z_p}({\cal F})\to {\cal H}^p_{Z_p/Z_{p+1}}({\cal F})\to {\cal H}^{p+1}_{Z_{p+1}}({\cal F})\to 0
\]
Since it is true for all $p\geq i$, we get that $\underline{{\cal C}\textup{ousin}}_{\{Z\},i}{\cal F}$ is a resolution of ${\cal H}^i_{Z_i}({\cal F})$.
Moreover, since $Z_k\setminus Z_{k+1}$ is affine, $H^j(X,{\cal H}_{Z_k/Z_{k+1}}({\cal F}))=0$ for $j>0$, hence $\underline{{\cal C}\textup{ousin}}_{\{Z\},i}{\cal F}$ is an acyclic resolution of ${\cal H}^i_{Z_i}({\cal F})$. This proves (2).
\end{proof}
When $X=G/P$ is a flag variety of dimension $n$, and if $B\subset P$ is a Borel subgroup, define for all $w\in W^P$, the Bruhat cell $X_w=BwP/P$, which is of dimension $l(w)$. Take $Z_i$ to be the union of all Bruhat cells of codimension at least $i$, then $Z_i\setminus Z_{i+1}$ is affine, and the complex $\textup{Cousin}_{\{Z\}}{\cal F}$ is a resolution of $H^0(X,{\cal F})$. When ${\cal F}={\cal L}_{\lambda}$ is an invertible sheaf with nonzero global sections, this resolution is actually dual to the BGG resolution of the irreducible highest weight module $V_{\lambda}$. If $X_w$ is a Bruhat cell in $X$, we will use the same idea to construct Cousin complexes $\textup{Cousin}_{\{Z\},w}{\cal F}$ and $\underline{{\cal C}\textup{ousin}}_{\{Z\},w}{\cal F}$ by restricting our filtration to $\overline{X_w}$, and we will look how close $\textup{Cousin}_{\{Z\},w}{\cal F}$ is to be a resolution of $H^i_{\overline{X_w}}(X,{\cal F})$.\\
\indent
Since $Z_i\setminus Z_{i+1}$ is the disjoint union of all $B$-orbits of codimension $i$, we have
\[
\begin{array}{rcl}
H^i_{Z_i/Z_{i+1}}(X,{\cal F})&=&H^i_{Z_i\setminus Z_{i+1}}(X\setminus Z_{i+1},{\cal F})\\
&=&\bigoplus\limits_{l(w)=n-i}H^i_{X_w}(X\setminus Z_{i+1},{\cal F})\\
&=&\bigoplus\limits_{l(w)=n-i}H^i_{X_w}(X,{\cal F})\\
\end{array}
\]
by excision lemma. Hence morphisms in $\textup{Cousin}_{\{Z\}}{\cal F}$ are 
\[
d^i:\bigoplus\limits_{l(w)=n-i}H^i_{X_w}(X,{\cal F})\to\bigoplus\limits_{l(w')=n-i-1}H^{i+1}_{X_{w'}}(X,{\cal F})
\]
and the morphism $H^i_{X_w}(X,{\cal F})\to H^{i+1}_{X_{w'}}(X,{\cal F})$ is zero when $X_{w'}\not\subset\overline{X_w}$. Let $X_w$ be a Bruhat cell of codimension $i$ in $X$. Then we have
\[
\bigoplus\limits_{\substack{l(w')=n-i-j \\ X_{w'}\subset \overline{X_w}}}H^{i+j}_{X_{w'}}(X,{\cal F})= H^{i+j}_{\overline{X_w}\cap Z_{i+j}/\overline{X_w}\cap Z_{i+j+1}}(X,{\cal F})
\]
and we have two complexes
\[
0\to H^i_{\overline{X_w}}(X,{\cal F})\overset{d_0}\to H^i_{X_w}(X,{\cal F})\overset{d_1}\to\bigoplus\limits_{\substack{l(w')=n-i-1 \\ X_{w'}\subset \overline{X_w}}}H^{i+1}_{X_{w'}}(X,{\cal F})\overset{d_2}\to\ldots
\]
and
\[
0\to {\cal H}^i_{\overline{X_w}}({\cal F})\overset{d_0}\to {\cal H}^i_{X_w}({\cal F})\overset{d_1}\to\bigoplus\limits_{\substack{l(w')=n-i-1 \\ X_{w'}\subset \overline{X_w}}}{\cal H}^{i+1}_{X_{w'}}({\cal F})\overset{d_2}\to\ldots
\]
which will be respectively denoted by $\textup{Cousin}_{\{Z\},w}{\cal F}$ and $\underline{{\cal C}\textup{ousin}}_{\{Z\},w}{\cal F}$.
\begin{theoreme}
Let ${\cal F}$ be a locally free sheaf on $X$, $i>0$ and $X_w$ a Bruhat cell of codimension $i$. Then
\item[(1)] we have an isomorphism $\textup{Cousin}_{\{Z\},w}{\cal F}\to H^0_{\overline{X_w}}(X,\underline{{\cal C}\textup{ousin}}_{\{Z\},w}{\cal F})$ ;
\item[(2)] the $j$-th homology group of the complex $\textup{Cousin}_{\{Z\},w}{\cal F}$ is isomorphic to $H^j_{\overline{X_w}}(X,{\cal H}^i_{\overline{X_w}}{\cal F})$.
\end{theoreme}
\begin{proof}
Let $k\geq i$. Since ${\cal H}^j_{Z_k/Z_{k+1}}({\cal F})=0$ for $j\neq k$, and if we denote by $j:X\setminus Z_{k+1}\to X$ the inclusion, since we have a spectral sequence
\[
\bigoplus\limits_{l(w')=n-k}R^pj_*{\cal H}^q_{\overline{X_{w'}}\cap Z_k\setminus\overline{X_{w'}}\cap Z_{k+1}}({\cal F}_{|X\setminus Z_{k+1}})=R^pj_*{\cal H}^q_{Z_k\setminus Z_{k+1}}({\cal F}_{|X\setminus Z_{k+1}})\Rightarrow {\cal H}^*_{Z_k/Z_{k+1}}({\cal F})
\]
we get that ${\cal H}^{l}_{\overline{X_w}\cap Z_k/\overline{X_w}\cap Z_{k+1}}({\cal F})=0$ for $l<k$. The spectral sequence
\[
H^p_{\overline{X_w}}(X,{\cal H}^q_{\overline{X_w}\cap Z_k/\overline{X_w}\cap Z_{k+1}}({\cal F}))\Rightarrow H^*_{\overline{X_w}\cap Z_k/\overline{X_w}\cap Z_{k+1}}(X,{\cal F})
\]
gives us
\[
H^k_{\overline{X_w}\cap Z_k/\overline{X_w}\cap Z_{k+1}}(X,{\cal F})=H^0_{\overline{X_w}}(X,{\cal H}^k_{\overline{X_w}\cap Z_k/\overline{X_w}\cap Z_{k+1}}({\cal F}))
\]
Moreover, since $\overline{X_w}$ is closed and ${\cal F}$ is locally free, we have ${\cal H}^p_{\overline{X_w}}({\cal F})=0$ for $p<i$, and we have an isomorphism
\[
H^i_{\overline{X_w}}(X,{\cal F})=H^0_{\overline{X_w}}(X,{\cal H}^i_{\overline{X_w}}({\cal F}))
\]
This proves (1).\\
\indent
To prove (2), we will first use the fact that $\underline{{\cal C}\textup{ousin}}_{\{Z\},i}{\cal F}$ is a resolution of ${\cal H}^i_{Z_i}({\cal F})$, and that the functor $\underline{\Gamma}_{\overline{X_w}}(\bullet)$ is left-exact. Since we have a spectral sequence
\[
{\cal H}^q_{\overline{X_w}}({\cal H}^{p}_{Z_{i+k}/Z_{i+k+1}}({\cal F}))\Rightarrow {\cal H}^*_{\overline{X_w}\cap Z_{i+k}/\overline{X_w}\cap Z_{i+k+1}}({\cal F})
\]
and since $Z_{i+k}\setminus Z_{i+k+1}$ and $\overline{X_w}\cap Z_{i+k}\setminus \overline{X_w}\cap Z_{i+k+1}$ are affine, we get that ${\cal H}^0_{\overline{X_w}}({\cal H}^{i+k}_{Z_{i+k}/Z_{i+k}}({\cal F}))={\cal H}^{i+k}_{\overline{X_w}\cap Z_{i+k}/\overline{X_w}\cap Z_{i+k+1}}({\cal F})$, and that the ${\cal H}^{i+k}_{Z_{i+k}/Z_{i+k}}({\cal F})$ are $\underline{\Gamma}_{\overline{X_w}}(\bullet)$-acyclic. Moreover, the spectral sequence
\[
{\cal H}^q_{\overline{X_w}}({\cal H}^{p}_{Z_{i}}({\cal F}))\Rightarrow {\cal H}^*_{\overline{X_w}\cap Z_{i}}({\cal F})
\]
gives us ${\cal H}^0_{\overline{X_w}}({\cal H}^i_{Z_i}({\cal F}))$. Hence the sequence
\[
0\to {\cal H}^i_{\overline{X_w}\cap Z_i}({\cal F})\to {\cal H}^i_{\overline{X_w}\cap Z_i/\overline{X_w}\cap Z_{i+1}}({\cal F})\to {\cal H}^{i+1}_{\overline{X_w}\cap Z_{i+1}/\overline{X_w}\cap Z_{i+2}}({\cal F})\to \ldots\to {\cal H}^n_{\overline{X_w}\cap Z_n}({\cal F})\to 0
\]
is exact. The same argument as previously and the spectral sequence
\[
H^q_{\overline{X_w}}(X,{\cal H}^{p}_{Z_{i+k}/Z_{i+k+1}}({\cal F}))\Rightarrow H^*_{\overline{X_w}\cap Z_{i+k}/\overline{X_w}\cap Z_{i+k+1}}(X,{\cal F})
\]
gives us that the ${\cal H}^{i+k}_{Z_{i+k}/Z_{i+k}}({\cal F})$ are $\Gamma_{\overline{X_w}}(\bullet)$-acyclic, hence (2).
\end{proof}
\subsection{Structure of $G$-module}
In \cite{Ke1} (11), Kempf has proved that if ${\cal F}$ is quasi-coherent and $G$-linearized on an algebraic variety $X$, the local cohomology groups (resp. the global Cousin complexes) have natural $\mathfrak{g}$-module (resp. complex of $\mathfrak{g}$-modules) structure. It is still true for the Cousin complex
$\textup{Cousin}_{\{Z\},w}$, since the map 
\[
H^{n-l(w)}_{\overline{X}_w}(X,{\cal F})\to H^{n-l(w)}_{\overline{X_w}/(\overline{X_w}\cap Z_{n-l(w)+1})}(X,{\cal F})
\]
is a $\mathfrak{g}$-modules morphism. If $V$ is a $\mathfrak{g}$-module, if $T\subset G$ is a maximal torus, and if $\lambda\in {\cal X}(T)$, let $V_{\lambda}=\{v\in V|\forall t\in T, t.v=\lambda(t).v\}$. When all the $V_{\lambda}$ are finite dimensional, let $[V]=\sum \dim V_{\lambda} e^{\lambda}$ be the character of $V$. If $X=G/P$ is a flag variety with $\textrm{Pic}(X)=\mathbb{Z}$, and if ${\cal L}$ is a $G$-linearized invertible sheaf on $X$, we can approximate the character of the $\mathfrak{g}$-module $H^i_{\overline{X_w}}(X,{\cal L})$ thanks to the Cousin complex $\textup{Cousin}_{\{Z\},w}({\cal L})$ :
\begin{proposition}
Let $X_w$ be a Bruhat cell in $X$ of codimension $i$, and let $j$ be the first positive integer (if it exists) such that $H^j_{\overline{X_w}}(X,{\cal H}^i_{\overline{X_w}}{\cal L})$ is nonzero. Then
\[
\begin{tabular}{rcl}
$[H^i_{\overline{X_w}}(X,{\cal L})]$&$=$&$\sum\limits_{k=0}^{j-1}(-1)^k[H^{i+k}_{Z_{i+k}\cap\overline{X_w}/Z_{i+k+1}\cap\overline{X_w}}(X,{\cal L})]+(-1)^j[\textrm{im }(d_{j})]$\\
&$=$&$\sum\limits_{k=0}^{j-1}(-1)^k(\sum\limits_{\substack{l(w')=n-i-k \\ X_{w'}\subset \overline{X_w}}}[H^{i+k}_{X_{w'}}(X,{\cal L})])+(-1)^j[\textrm{im }(d_{j})]$
\end{tabular}
\]
and for all weight $\lambda$, we have $[\textrm{im }(d_{j})](\lambda)\leq [H^{i+j}_{Z_{i+j}\cap\overline{X_w}/Z_{i+j+1}\cap \overline{X_w}}(X,{\cal L})](\lambda)$
\end{proposition}
\begin{proof}
This is true because the sequence
\[
\begin{aligned}
0&\to H^i_{\overline{X_w}}(X,{\cal F})\to H^i_{\overline{X_w}/\overline{X_w}\cap Z_{i+1}}(X,{\cal F})\to\ldots \\
&\to H^{j-1}_{Z_{j-1}\cap\overline{X_w}/Z_j\cap\overline{X_w}}(X,{\cal F})\to H^{j}_{Z_{j}\cap\overline{X_w}/Z_{j+1}\cap\overline{X_w}}(X,{\cal F})\to H^{j}_{Z_{j}\cap\overline{X_w}/Z_{j+1}\cap\overline{X_w}}(X,{\cal F})/\textrm{im }(d_{j})\to 0
\end{aligned}
\]
is exact, and because $\textrm{im }(d_{j})$ is a sub-object of $H^{i+j}_{Z_{i+j}\cap\overline{X_w}/Z_{i+j+1}\cap \overline{X_w}}(X,{\cal L})$.
\end{proof}
The interesting point is that we can compute the characters $[H^{\textrm{codim}(X_{w'})}_{X_{w'}}(X,{\cal L})]$ thanks to \cite{Ke1} 12.8. Let $K(w)$ be the set of positive roots in $wR^u(P)w^{-1}$, let $L(w)$ be the set of the opposite of the negative roots in $wR^u(P)w^{-1}$, and let $J(w)=K(w)\sqcup L(w)$. Then we have
\begin{proposition}
$H^j_{X_w}(X,{\cal L}_{\lambda})=0$ for $j\neq \textrm{codim}(X_w)$, and
\[
[H^{\textrm{codim}(X_w)}_{X_w}(X,{\cal L}_{\lambda})]=e^{ww_{0,P'}(\lambda)}\frac{\prod\limits_{\alpha\in K(w)}e^{\alpha}}{\prod\limits_{\beta\in J(w)}(1-e^{\beta})}
\]
\end{proposition}
\subsection{Relations with cohomology on geometric quotients}
Let $X$ be a smooth projective irreducible algebraic variety acted on by a reductive algebraic group $G$, and let ${\cal M}$ be a $G$-linearized invertible sheaf on $X$. Assume the quotient $\pi:X^{ss}({\cal M})\to Y:=X/G$ is geometric. Let $X^us=X\setminus X^{ss}({\cal M})$. If ${\cal L}$ is an invertible sheaf on $Y$, we can construct a $G$-linearized invertible sheaf $\overline{\cal L}$ on $X^{ss}({\cal M})$ such that for all open $U\subset Y$, ${\cal L}(U)=\overline{\cal L}(\pi^{-1}(U))^G$, and this is done by pulling back a Cartier divisor on $Y$ to $X^{ss}({\cal M})$.\\
\indent
If ${\cal U}=(U_i)$ is a open cover of $Y$ by affine open subsets trivializing ${\cal L}$, the open cover ${\cal V}=(\pi^{-1}(U_i))$ of $X^{ss}({\cal M})$ is a cover by affine open $G$-invariant subsets trivializing $\overline{\cal L}$. We have two \v{C}ech complexes $C^*({\cal U},{\cal L})$ (with maps denoted by $d_k$) and $C^*({\cal V},\overline{\cal L})$ (with maps denoted by $d'_k$), and we have $C^*({\cal U},{\cal L})=C^*({\cal V},\overline{\cal L})^G$. It is obvious that $\ker d'_k=\ker d_k\cap C^k({\cal U},{\cal L})$, and we have
\begin{proposition}
$\textup{im} d'_k=\textup{im} d_k\cap C^{k+1}({\cal U},{\cal L})$ and $\check{H}^k({\cal U},{\cal L})=\check{H}^k({\cal V},\overline{\cal L})^G$.
\end{proposition}
\begin{proof}
Since $G$ is reductive, $C^k({\cal V},\overline{\cal L})^G$ has a supplement in $C^k({\cal V},\overline{\cal L})$. Hence we have projections $p_k: C^k({\cal V},\overline{\cal L})\to C^k({\cal V},\overline{\cal L})^G$, and the diagram
\begin{center}
\begin{tikzcd}
C^{k}({\cal V},\overline{\cal L}) \arrow{r}{d_k} \arrow{d}{p_k}
		&C^{k+1}({\cal V},\overline{\cal L}) \arrow{d}{p_{k+1}}\\
C^{k}({\cal V},\overline{\cal L})^G \arrow{r}[below]{d'_k}
		&C^{k+1}({\cal V},\overline{\cal L})^G
\end{tikzcd}
\end{center}
commutes, hence $\textup{im} d'_k=\textup{im} d_k\cap C^{k+1}({\cal U},{\cal L})$ and $\check{H}^k({\cal U},{\cal L})=\check{H}^k({\cal V},\overline{\cal L})^G$.
\end{proof}
Hence we have isomorphisms $H^k(Y,{\cal L})\simeq H^k(X^{ss}({\cal M}),\overline{\cal L})^G$. We have a long exact sequence
\[
\ldots\to H^{i-1}(X^{ss}({\cal M}),\overline{\cal L})\to H^i_{X^{us}}(X,\overline{\cal L})\to H^i(X,\overline{\cal L})\to H^i(X^{ss}({\cal M}),\overline{\cal L})\to H^{i+1}_{X^{us}}(X,\overline{\cal L})\to\ldots
\]
This sequence gives us $H^k(X^{ss}({\cal M}),\overline{\cal L})$ in some cases :\\
\textbullet \indent if $H^k(X,\overline{\cal L})=H^{k+1}(X,\overline{\cal L})=0$, we have an isomorphism
\[
H^k(X^{ss}({\cal M}),\overline{\cal L})\simeq H^{k+1}_{X^{us}}(X,\overline{\cal L})
\]
\textbullet \indent if $k<\textrm{codim} X^{us}-1$, then
\[
H^k(X^{ss}({\cal M}),\overline{\cal L})\simeq H^k(X,\overline{\cal L})
\]
Moreover, $X^{us}$ is $G$-stable, hence these isomorphisms are isomorphisms of $G$-modules, and we have the same isomorphisms between their $G$-invariants.\\
\indent
When $X^{us}({\cal M})$ has codimension at least two, we can compare invertible sheaves on $Y$ and $G$-linearized invertible sheaves on $X$ :
\begin{proposition}
Let $j:X^{ss}({\cal M})\to X$ be the inclusion, and assume that $X^{us}({\cal M})$ has codimension at least two. Then the morphism
\[
\begin{array}{rcl}
q:\textrm{Pic}^G(X)&\to &\textrm{Pic}(Y)\\
{\cal L}&\mapsto&(\pi_*j^*{\cal L})^G
\end{array}
\]
admits a cosection $s:\textrm{Pic}(Y)\to\textrm{Pic}^G(X)$, which is a group morphism.
\end{proposition}
\begin{proof}
Let ${\cal L}\in\textrm{Pic}(Y)$. There is a Cartier divisor $D=(f_i)$ such that ${\cal L}={\cal O}_Y(D)$. Then $(\pi^*f_i)$ is a Cartier divisor on $X^{ss}(F)$, and since $X^{ss}(F)$ is of codimension at least two, the $\pi^*f_i$ can be uniquely extended to rational functions $\tilde{f_i}$ on $X$. Let $\tilde{D}=(\tilde{f_i})$, and $s({\cal L})={\cal O}_X(\tilde{D})$, with a $G$-linearization given by saying the $\tilde{f_i}$'s are $G$-invariant. Then $q(s({\cal L}))={\cal L}$, and since $(\pi^*f_i)(\pi^*g_j)=\pi^*(f_ig_j)$, $s$ is a group morphism.
\end{proof}
\subsection{Case of $\overline{\textrm{PGL}_3}$}
Before going back to our examples of type $A_2$, let us give some results.\\
If $V$ is a $\mathfrak{g}$-module and if $g\in G$, we will denote by $gV$ the $\mathfrak{g}$-module $V$ with the action $\xi._{g}v=ad(g)(\xi).v$.
\begin{proposition}
If $X$ is a $G$-variety, if $Z\subset X$ is closed, and if ${\cal L}$ is $G$-linearized, for all $g\in G$, we have an isomorphism of $\mathfrak{g}$-modules $H^i_{gZ}(X,{\cal L})=g^{-1}H^i_Z(X,g{\cal L})$.
\end{proposition}
\begin{proof}
It is enough to show it for $i=0$ :
\[
\begin{aligned}
\Gamma_{gZ}(X,{\cal L})&=\{\sigma\in\Gamma(X,{\cal L}),\sigma_{|gZ}=0\}\\
&=\{g^{-1}\sigma,\sigma\in\Gamma_Z(X,g{\cal L}),\sigma_{|Z}=0\}\\
&=g^{-1}\Gamma_Z(X,g{\cal L})
\end{aligned}
\]
\end{proof}
Moreover, when $X=G/P$ is a flag variety, we show :
\begin{proposition}
If $Z\subset X$ is closed and connected of codimension $l$, then for all invertible sheaf ${\cal L}$ on $X$, $H^l_{Z}(X,{\cal L})$ is simple as a left $D_{X,{\cal L}}$-module.
\end{proposition}
\begin{proof}
Let $X^{an}$ be the analytic space associated to $X$, and $j:X^{an}\to X$ the natural map. Let $\underline{\Gamma}_{[Z]}(\bullet)$ and ${\cal H}^*_{[Z]}(\bullet)$ be the analytic equivalents of the functors $\underline{\Gamma}_{Z}(\bullet)$ and ${\cal H}^*_{Z}(\bullet)$. Let us recall the two following facts (cf. \cite{BK}) :\\
\textbullet \indent for all coherent ${\cal O}_X$-module, ${\cal H}^*_{[Z]}(j^*{\cal F})=j^*{\cal H}^*_Z({\cal F})$\\
\textbullet \indent if ${\cal M}$ is a coherent ${\cal D}_{X^{an}}$-module whose characteristic variety is contained in the characteristic variety of ${\cal H}^l_{[Z]}({\cal O}_{X^{an}})$, the sheaf ${\cal H}\textup{om}_{{\cal D}_{X^{an}}}({\cal H}^l_{[Z]}({\cal O}_{X^{an}}),{\cal M})$ is locally constant of finite rank, and we have an isomorphism
\[
{\cal H}^l_{[Z]}({\cal O}_{X^{an}})\otimes_{\mathbb{C}}{\cal H}\textup{om}_{{\cal D}_{X^{an}}}({\cal H}^l_{[Z]}({\cal O}_{X^{an}}),{\cal M})\simeq {\cal M}
\]
\indent
Now assume $Z$ is connected. If ${\cal M}$ is a coherent left sub-${\cal D}_{X^{an}}$-module of ${\cal H}^l_{[Z]}({\cal O}_{X^{an}})$, the sheaf ${\cal H}\textup{om}_{{\cal D}_{X^{an}}}({\cal H}^l_{[Z]}({\cal O}_{X^{an}}),{\cal M})$ is on $Z$ the sheaf 0 or $\underline{\mathbb{C}}$, hence ${\cal M}=0$ or ${\cal H}^l_{[Z]}({\cal O}_{X^{an}})$.\\
\indent
Let ${\cal L}$ be an invertible sheaf on $X$, and let ${\cal F}$ be a quasi-coherent left sub-${\cal D}_X$-module of ${\cal H}^l_Z({\cal L})={\cal H}^l_Z({\cal O}_X)\otimes {\cal L}$. Then $j^*({\cal F}\otimes{\cal L}^{\otimes -1})$ is a quasi-coherent sub-${\cal D}_{X^{an}}$-module of ${\cal H}^l_{[Z]}({\cal O}_{X^{an}})$. Since quasi-coherent sheaves are direct limits of their coherent subsheaves, we get that $j^*({\cal F}\otimes{\cal L}^{\otimes -1})=0$, hence ${\cal F}=0$.\\
\indent
Now let $M$ be a strict left sub-$D_{X,{\cal L}}$-module of $H^l_Z(X,{\cal L})=H^0(X,{\cal H}^l_Z({\cal L}))$. Since $X$ is a flag variety, it is ${\cal D}$-affine, and ${\cal D}_{X,{\cal L}}\otimes_{D_{X,{\cal L}}} M$ is a strict quasi-coherent left sub-${\cal D}_X$-module of ${\cal H}^l_Z({\cal L})$, hence $M=0$. Hence $H^l_{z}(X,{\cal L})$ is simple as a left $D_{X,{\cal L}}$-module.
\end{proof}
We also need a way to compare left and right actions of ${\cal D}_X$. The following is proved in \cite{HTT} (1.2.7) :
\begin{proposition}
If $A$ is a ring, let $A^{op}$ be the ring such that $A=A^{op}$ as abelian groups, and the multiplication $\star$ on $A^{op}$ is given by $a\star b=ba$. We have an isomorphism
\[
{\cal D}_X^{op}\simeq \omega_X\otimes{\cal D}_X\otimes\omega_X^{\otimes -1}
\]
where $\omega_X$ is the canonical sheaf of $X$.
\end{proposition}
In particular, ${\cal D}_{X,{\cal L}}^{op}={\cal D}_{X,{\cal L}^{\otimes -1}\otimes\omega_X}$, and we have a equivalence of categories between the quasi-coherent left ${\cal D}_{X,{\cal L}}$-modules and quasi-coherent right ${\cal D}_{X,{\cal L}^{\otimes -1}\otimes\omega_X}$-modules, and when $X$ is a flag variety. We also have an equivalence of categories between the left-$D_{X,{\cal L}}$-modules and the right $D_{X,{\cal L}^{\otimes -1}\otimes\omega_X}$-modules.\\
\indent
We will now look at the cohomology group $H^i(Y,{\cal L})$ when $Y=\overline{\textrm{PGL}_3}$, the cases of $\overline{\textrm{PGL}_3/\textrm{PSO}_3}$ and of $\overline{\textrm{PGL}_6/\textrm{PSp}_6}$ being similar. Let $G',T',B',P',\pi_0$ and $\iota$ as in the section 3.2. Let $\Lambda$ be the weight lattice of $T'$. Let us recall that $X^{us}$ has two connected components
\[
F_1=\{U\in X, \dim(U\cap V)\geq 2\} \textrm{ and } F_2=\{U\in X, \dim(U\cap V^*)\geq 2\}
\]
Let $U\in F_1$. Then for all $b\in B'$, $\dim bU\cap V=\dim U\cap b^{-1}V=\dim U\cap V\geq 2$, hence $F_1$ is $B'$-stable. Since it is closed, it is the closure of a Bruhat cell $X_w$. Since $V^*=w_{0,P'}.V$ (where $w_{0,P'}$ is the longest element in the Weyl group $W^{P'}=W/W_{P'}$), $F_2=w_{0,P'}.F_1$. To find which $\overline{X_w}$ is $F_1$, let us look at its $T'$-fixed points. Since $W^{P'}$ acts by permutations on the basis elements $\{e_1,e_2,e_3,e_1^*,e_2^*,e_3^*\}$, $w.V$ is in $F_1$ if and only if at least two of the $w(e_1),w(e_2),w(e_3)$ are in $\{e_1,e_2,e_3\}$. We will write $s_{i_1\ldots i_k}$ instead of $s_{\alpha_{i_1}}\ldots s_{\alpha_{i_k}}$. The $T'$-fixed points are given by
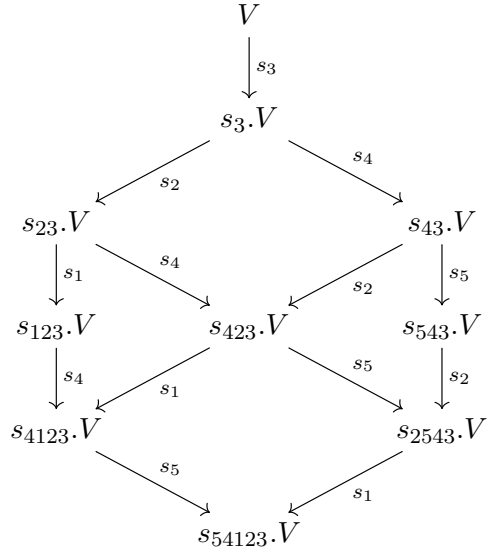
\begin{center}
\begin{tikzcd}
{}											&V \arrow{d}{s_3}								&\\
											&s_3.V \arrow{ld}{s_2} \arrow{rd}{s_4}			&\\
s_{23}.V \arrow{d}{s_1} \arrow{rd}{s_4}		&												&s_{43}.V \arrow{d}{s_5}\arrow{ld}{s_2}\\
s_{123}.V \arrow{d}{s_4}						&s_{423}.V \arrow{ld}{s_1}\arrow{rd}{s_5}		&s_{543}.V \arrow{d}{s_2}\\
s_{4123}.V \arrow{rd}{s_5}					&												&s_{2543}.V \arrow{ld}{s_1}\\
											&s_{54123}.V
\end{tikzcd}
\end{center}
Then if we let $w=s_{3423}$, we have $X_{w}=B's_{54123}.V$, and $F_1=\overline{X_w}$.\\
\indent
Let us recall that since $Y$ is a wonderful variety of minimal rank, Tchoudjem gave in \cite{Tc1} a description of the cohomology groups $H^i(Y,{\cal L})$ as $\mathfrak{g}$-modules given by
\[
H^i(Y,{\cal L}_{\lambda})\simeq \bigoplus\limits_{J\subset\Sigma_X}\bigoplus\limits_{\substack{\mu\in(\lambda+R_J)\cap\Omega_J\\ \mu+\rho\textrm{ régulier}\\l(\mu)+|J|=i}}V_{\mu^+}^*
\]
where for all $J\subset\Sigma_X$,
\[
R_J:=\sum\limits_{\gamma\in J}\mathbb{Z}_{>0}\gamma+\sum\limits_{\gamma\in\Sigma_X\setminus J}Z_{\leq 0}\gamma
\]
and
\[
\Omega_J:=\{\mu\in\Gamma|\{\gamma\in\Sigma_X|(\mu+\rho,\gamma)<0\}=J\}
\]
In particular, the cohomology groups $H^i(Y,{\cal L})$ can be nonzero only for $i=0,3,5$ or 8.
\begin{remark}
$\overline{\textrm{PGL}_3/\textrm{PSO}_3}$ is not a wonderful variety of minimal rank, but Tchoudjem has shown in \cite{Tc2} that the cohomology groups $H^i(Y,{\cal L})$ have a similar description.
\end{remark}
\indent
Let $\tilde{\varpi_i}=(\varpi_i,\varpi_i')$ for $i=1,2$. Let us recall that if $U\subset Y$ is open, we denote by $\overline{\cal L}(\pi_0^{-1}(U))_n$ the part of degree $n$ for the $\mathbb{C}^*$-action of $\overline{\cal L}(\pi_0^{-1}(U))$. We can describe more precisely what the cosection $s$ is with the following proposition :
\begin{proposition}
Up to renumbering, for all open $U\subset Y$ :
\[
\begin{array}{rcl}
{\cal L}_{\tilde{\varpi_1}}(U)&=&\overline{\cal L}_{\varpi_3}(\pi_0^{-1}(U))_{-1}\\
{\cal L}_{\tilde{\varpi_2}}(U)&=&\overline{\cal L}_{\varpi_3}(\pi_0^{-1}(U))_{1}\\
\end{array}
\]
\end{proposition}
\begin{proof}
We already know there are $k_i$ and $n_i$ such that
\[
{\cal L}_{\tilde{\varpi_1}}(U)=\overline{\cal L}_{k_i\varpi_3}(\pi_0^{-1}(U))_{n_i}
\]
We have $H^0(Y,{\cal L}_{\tilde{\varpi_i}})=V_{\tilde{\varpi_i}}^*\neq 0$, hence $k_i\varpi_3$ is dominant, and $k_i\geq 0$. Since $\overline{\cal L}_0={\cal O}_X$, its global sections are isomorphic to $\mathbb{C}$, and the $\mathbb{C}^*$-semi-invariants of its global sections are concentrated in degree 0. Since for all open $U\subset Y$, ${\cal O}_Y(U)={\cal O}_X(\pi_0^{-1}(U))_0$, we have $k_i>0$.\\
\indent
Let $\Lambda'$ be the weight lattice of $\textrm{SL}_3\times\textrm{SL}_3$. Let us remark that the set of $\lambda\in\Lambda'$ such that $H^0(Y,{\cal L}_{\lambda})\neq 0$, which will be denoted by $SG(\Lambda')$, is the convex cone spanned by $\gamma_1$ and $\gamma_2$ intersected with $\Lambda '$. For all $\lambda\in SG(\Lambda')$, there exists non-negative integers $a_1,a_2,b_1,b_2$ such that $\lambda=a_1\tilde{\varpi_1}+a_2\tilde{\varpi_2}+b_1\gamma_1+b_2\gamma_2$. Moreover, for all $\lambda\in SG(\Lambda')\setminus 0$, we have
\[
{\cal L}_{\lambda}(U)=\overline{\cal L}_{k_{\lambda}\varpi_3}(\pi_0^{-1}(U))_{n_{\lambda}}
\]
with $n_{\lambda}\in\mathbb{Z}$ and $k_{\lambda}>a_1k_1+a_2k_2+b_1k_{\gamma_1}+b_2k_{\gamma_2}>0$, hence $k_{\lambda}=1$ is possible only for $\lambda=\tilde{\varpi_1}, \tilde{\varpi_2}, \gamma_1$ or $\gamma_2$. Let $k'_i=k_{\gamma_i}$, and $n'_i=n_{\gamma_i}$. Since $H^0(X,\overline{\cal L}_{\varpi_3})=\Lambda^3(\mathbb{C}^3\oplus\mathbb{C}^3)$ where $\mathbb{C}^*$ acts with weight 1 on the first copy of $\mathbb{C}^3$ and -1 on the second one, its $\mathbb{C}^*$-semi-invariants have weight -3, -1, 1 or 3. In particular, if $\mathbb{C}^*$ acts with one of these weights on trivializing sections of $\overline{\cal L}_{\varpi_3}$, $q(\overline{\cal L}_{\varpi_3})\neq{\cal O}_Y$. Hence  at least one of the $k_i$ or $k'_i$ equals 1. Since $k'_i=2k_i-k_j$ with $i\neq j$, if one of the $k'_i$ or $k_i$ equals 1, they all are equal to 1. Moreover, $n'_i=2n_i-n_j$ and $\{n_1,n_2,n'_1,n'_2\}=\{-3;-1;1;3\}$, hence up to renumbering, we have  $n_1=-1,n_2=1,n'_1=-3$ et $n'_2=3$.
\end{proof}
Let us remark that since cohomology groups $H^i(X,\overline{\cal L})$ can be nonzero only when $i=0$ or 9, we have isomorphisms
\[
H^3(Y,{\cal L})=H^4(X^{ss}(0),s({\cal L}))^{\mathbb{C}^*}=H^4_{X^{us}}(X,s({\cal L}))^{\mathbb{C}^*}
\]
Since $X^{us}$ is a disjoint union of the two closed subsets $F_1$ and $F_2$, we have
\[
H^4_{X^{us}}(X,\overline{\cal L})=H^4_{F_1}(X,\overline{\cal L})\oplus H^4_{F_2}(X,\overline{\cal L})
\]
$H^4_{F_1}(X,\overline{\cal L}_{k\varpi_3})$ is the kernel of
\[
H^4_{X_w}(X,\overline{\cal L}_{k\varpi_3})\to H^5_{X_{s_1w}}(X,\overline{\cal L}_{k\varpi_3})\oplus H^5_{X_{s_5w}}(X,\overline{\cal L}_{k\varpi_3})
\]
Hence for all $\lambda\in\Lambda$
\[
\begin{tabular}{rcl}
$[H^4_{X_w}(X,\overline{\cal L}_{k\varpi_3})](\lambda)$&$\geq $&$[H^4_{F_1}(X,\overline{\cal L}_{k\varpi_3})](\lambda)$\\
$[H^4_{F_1}(X,\overline{\cal L}_{k\varpi_3})](\lambda)$&$\geq $&$[H^4_{X_w}(X,\overline{\cal L}_{k\varpi_3})](\lambda)-[H^5_{X_{s_1w}}(X,\overline{\cal L}_{k\varpi_3})](\lambda)-[H^5_{X_{s_5w}}(X,\overline{\cal L}_{k\varpi_3})](\lambda)$
\end{tabular}
\]
We will now compute these characters. We have :\\
\textbullet \indent $ww_{0,P'}(k\varpi_3)=\frac{k}{2}(\alpha_3-\alpha_5-\alpha_1)$ ;\\
\textbullet \indent $s_1ww_{0,P'}(k\varpi_3)=\frac{k}{2}(\alpha_3-\alpha_5+\alpha_1)$ ;\\
\textbullet \indent $s_5ww_{0,P'}(k\varpi_3)=\frac{k}{2}(\alpha_3+\alpha_5-\alpha_1)$ ;\\
\textbullet \indent $K(w)=\{\alpha_2+\alpha_3,\alpha_3,\alpha_3+\alpha_4,\alpha_2+\alpha_3+\alpha_4\}$ ;\\
\textbullet \indent $L(w)=\{\alpha_1,\alpha_1+\alpha_2,\alpha_4+\alpha_5,\alpha_5,\alpha_1+\alpha_2+\alpha_3+\alpha_4+\alpha_5\}$ ;\\
\textbullet \indent $K(s_1w)=\{\alpha_1,\alpha_1+\alpha_2+\alpha_3,\alpha_3,\alpha_1+\alpha_2+\alpha_3+\alpha_4,\alpha_3+\alpha_4\}$ ;\\
\textbullet \indent $L(s_1w)=\{\alpha_2,\alpha_4+\alpha_5,\alpha_5,\alpha_2+\alpha_3+\alpha_4+\alpha_5\}$ ;\\
\textbullet \indent $K(s_5w)=\{\alpha_2+\alpha_3,\alpha_3,\alpha_2+\alpha_3+\alpha_4+\alpha_5,\alpha_3+\alpha_4+\alpha_5,\alpha_5\}$ ;\\
\textbullet \indent $L(s_5w)=\{\alpha_1,\alpha_1+\alpha_2,\alpha_4,\alpha_1+\alpha_2+\alpha_3+\alpha_4\}$\\
and
\[
\begin{tabular}{rcl}
$[H^4_{X_w}(X,\overline{\cal L}_{k\varpi_3})]$&$=$&$e^{\frac{k}{2}(\alpha_3-\alpha_5-\alpha_1)}e^{\alpha_2+\alpha_3}e^{\alpha_3}e^{\alpha_3+\alpha_4}e^{\alpha_2+\alpha_3+\alpha_4}\prod\limits_{\lambda\in J(w)}\sum\limits_{k\geq 0}e^{\lambda}$\\
$[H^5_{X_{s_1w}}(X,\overline{\cal L}_{k\varpi_3})]$&$=$&$e^{\frac{k}{2}(\alpha_3-\alpha_5+\alpha_1)}e^{\alpha_1}e^{\alpha_1+\alpha_2+\alpha_3}e^{\alpha_3}e^{\alpha_3+\alpha_4}e^{\alpha_1+\alpha_2+\alpha_3+\alpha_4}\prod\limits_{\lambda\in J(s_1w)}\sum\limits_{k\geq 0}e^{\lambda}$\\
$[H^5_{X_{s_5w}}(X,\overline{\cal L}_{k\varpi_3})]$&$=$&$e^{\frac{k}{2}(\alpha_3+\alpha_5-\alpha_1)}e^{\alpha_2+\alpha_3}e^{\alpha_3}e^{\alpha_3+\alpha_4+\alpha_5}e^{\alpha_2+\alpha_3+\alpha_4+\alpha_5}e^{\alpha_5}\prod\limits_{\lambda\in J(s_5w)}\sum\limits_{k\geq 0}e^{\lambda}$
\end{tabular}
\]
Hence we can see that the $\mathbb{C}^*$-invariants of $H^4_{F_1}(X,\overline{\cal L}_{k\varpi_3})$ have weights bigger than $8+k$. Since we have $[H^i_{F_2}(X,\overline{\cal L})]=w_{0,P'}[H^i_Z(X,w_{0,P'}\overline{\cal L})]$, the $\mathbb{C}^*$-invariants of $H^4_{F_2}(X,\overline{\cal L}_{k\varpi_3})$ have weights smaller than $-8+k$. Hence for all invertible sheaf ${\cal L}$ on $Y$, at least one of the $H^4_{F_1}(X,s({\cal L}))^{\mathbb{C}^*}$ and $H^4_{F_2}(X,s({\cal L}))^{\mathbb{C}^*}$ is zero. Now we can prove
\begin{theoreme}
For all invertible sheaf ${\cal L}$ on $Y$ and for all $i$, $H^i(Y,{\cal L})$ is either 0 or simple as a left $D_{Y,{\cal L}}$-module.
\end{theoreme}
\begin{proof}
The case $i=0$ has already be treated in section 3.2. Assume for now that $i=3$, and that $H^i(Y,{\cal L})\neq 0$. There are two integers $k$ and $n$ such that for all open $U\subset Y$, we have
\[
{\cal L}(U)\simeq\overline{\cal L}_{k\varpi_3}(\pi_0(U))_n
\]
and
\[
H^3(Y,{\cal L})=H^3(X^{ss}(0),\overline{\cal L}_{k\varpi_3})_n=H^4_{X^{us}}(X,\overline{\cal L}_{k\varpi_3})_n
\]
We have just seen that there exists $i\in\{1,2\}$ such that $H^4_{X^{us}}(X,\overline{\cal L}_{k\varpi_3})_n=H^4_{F_i}(X,\overline{\cal L}_{k\varpi_3})_n$, and $F_i$ is connected. Hence $H^4_{F_i}(X,\overline{\cal L}_{k\varpi_3})$ is a simple as a left $D_{X,\overline{\cal L}_{k\varpi_3}}$-module, and $H^4_{F_i}(X,\overline{\cal L}_{k\varpi_3})_n$ is a simple left $D_{X,\overline{\cal L}_{k\varpi_3}}^{\mathbb{C}^*}$-module. In particular, $H^4_{F_i}(X,{\cal O}_X)_n$ is a simple left $D_{X}^{\mathbb{C}^*}$-module. Let $\Omega$ be a $\mathbb{C}^*$-invariant open subset of $X$ containing $F_i$. Since ${\cal H}^4_{F_i}(\overline{\cal L}_{k\varpi_3})$ has support in $F_i$, we have $H^4_{F_i}(X,\overline{\cal L}_{k\varpi_3})={\cal H}^4_{F_i}(\overline{\cal L}_{k\varpi_3})(\Omega)$. If ${\cal F}$ if a sheaf on $X$, let 
\[
{\cal F}^{glob}(\Omega):=\{\sigma_{|\Omega},\sigma\in\Gamma(X,{\cal F})\}
\]
Then ${\cal H}^4_{F_i}(\overline{\cal L}_{k\varpi_3})(\Omega)_n$ is a simple left ${\cal D}_X^{glob}(\Omega)_0$-module. Moreover, since $\overline{\cal L}_{k\varpi_3}\otimes {\cal D}_X\otimes\overline{\cal L}_{k\varpi_3}^{\otimes -1}$ is isomorphic to ${\cal D}_X$ as a ${\cal O}_X$-module, we have
\[
((\overline{\cal L}_{k\varpi_3})_n\otimes ({\cal D}_X)_0\otimes(\overline{\cal L}_{k\varpi_3}^{\otimes -1})_{-n})^{glob}(\Omega)=\overline{\cal L}_{k\varpi_3}(\Omega)_n\otimes {\cal D}_X^{glob}(\Omega)_0\otimes\overline{\cal L}_{k\varpi_3}^{\otimes -1}(\Omega)_{-n}
\]
Since
\[
{\cal H}^4_{F_i}(\overline{\cal L}_{k\varpi_3})={\cal H}^4_{F_i}({\cal O}_X)\otimes\overline{\cal L}_{k\varpi_3}
\]
we get that ${\cal H}^4_{F_i}(\overline{\cal L}_{k\varpi_3})(\Omega)_n$ is a simple left $\overline{\cal L}_{k\varpi_3}(\Omega)_n\otimes {\cal D}_X^{glob}(\Omega)_0\otimes\overline{\cal L}_{k\varpi_3}^{\otimes -1}(\Omega)_{-n}$-module, ie. $H^4_{F_i}(X,\overline{\cal L}_{k\varpi_3})_n$ is a simple left $H^0\Big(X,(\overline{\cal L}_{k\varpi_3})_n\otimes ({\cal D}_X)_0\otimes(\overline{\cal L}_{k\varpi_3}^{\otimes -1})_{-n}\Big)$-module. Since 
\[
H^4_{F_i}(X,\overline{\cal L}_{k\varpi_3})_n=H^3(Y,{\cal L})
\]
and 
\[
H^0(X,(\overline{\cal L}_{k\varpi_3})_n\otimes ({\cal D}_X)_0\otimes(\overline{\cal L}_{k\varpi_3}^{\otimes -1})_{-n})\simeq D_{Y,{\cal L}}\otimes \mathbb{C}[t\partial_{t}]
\]
and since the action of $t\partial_t$ on $H^3(Y,{\cal L})$ is trivial, $H^3(Y,{\cal L})$ is a simple left $D_{Y,{\cal L}}$-module.\\
\indent
If $i=5$ or 8, we will use the Serre duality
\[
H^i(Y,{\cal L}_{\lambda})\simeq H^{n-i}(Y,{\cal L}_{-\lambda-\mu})^*
\]
Since $H^{n-i}(Y,{\cal L}_{-\lambda-\mu})$ is a simple left $D_{Y,{\cal L}^{\otimes -1}\otimes \omega_Y}$-module, its dual is a simple right $D_{Y,{\cal L}^{\otimes -1}\otimes \omega_Y}$-module, hence it is a simple left $D_{Y,{\cal L}}$-module.
\end{proof}

\bibliographystyle{plain}

\end{document}